\documentclass[a4paper,12pt]{article}
\usepackage{amsmath}
\usepackage{pifont}
\usepackage{caption}
\usepackage{graphicx, subfig}
\usepackage{bm}
\usepackage{latexsym,amsmath,amssymb,cite,amsthm}
\usepackage{color,eucal,enumerate,mathrsfs}
\usepackage[normalem]{ulem}
\usepackage{amsmath}
\usepackage[pagewise]{lineno}	
\usepackage[pagebackref=true, colorlinks, linkcolor=blue,  anchorcolor=blue,citecolor=blue]{hyperref}
\usepackage{seqsplit}
\usepackage[titletoc,title]{appendix}

\numberwithin{equation}{section}
\theoremstyle{plain}
\newtheorem{theorem}{Theorem}[section]
\newtheorem{corollary}[theorem]{Corollary}
\newtheorem{lemma}[theorem]{Lemma}
\newtheorem{example}[theorem]{Example}
\newtheorem{proposition}[theorem]{Proposition}
\newtheorem{counterexample}[theorem]{Counterexample}
\theoremstyle{definition}
\newtheorem{definition}[theorem]{Definition}  

\newtheorem{assumption}[theorem]{Assumption}

\theoremstyle{remark}
\newtheorem{remark}[theorem]{Remark}

\newcommand\bdf{\begin{definition}}
	\newcommand\bpr{\begin{proposition}}
		\newcommand\brk{\begin{remark}}
			\newcommand\blm{\begin{lemma}}
				\newcommand\bexe{\begin{exercise}}
					\newcommand\bexa{\begin{example}}
						\newcommand\beqn{\begin{eqnarray*}}
							\newcommand\edf{\end{definition}}
						\newcommand\epr{\end{proposition}}
					\newcommand\erk{\end{remark}}
				\newcommand\elm{\end{lemma}}
			\newcommand\eexe{\end{exercise}}
		\newcommand\eexa{\end{example}}
	\newcommand\eeqn{\end{eqnarray*}}

\newcommand{\lip}[1]{{\mathrm{lip}}({#1})}
\newcommand{\mm}{\mathfrak m}

\newcommand{\vol}{{\rm Vol}_{ g}}

\newcommand{\N}{\mathbb{N}}

\newcommand{\R}{\mathbb{R}}

\newcommand{\diam}{\mathop{\rm diam}\nolimits} 

\newcommand{\supp}{\mathop{\rm supp}\nolimits}   %%\newcommand{\span}{\mathop{\rm span}\nolimits}   %
\newcommand{\Lip}{\mathop{\rm Lip}\nolimits}

\renewcommand{\d}{{\mathrm d}}

\newcommand{\D}{{\mathrm D}}
\newcommand{\restr}[1]{\lower3pt\hbox{$|_{#1}$}}

\newcommand{\nchi}{{\raise.3ex\hbox{$\chi$}}}

\newcommand{\MSC}[2][]{%
  \par\noindent\textbf{MSC}%
  \if\relax\detokenize{#1}\relax\else\space(#1)\fi%
  \textbf{:} #2\par
}

\title{{A new characterization of Sobolev spaces on Lipschitz differentiability spaces}
}

\begin{document}
	
	%	[中括号里面：输入显示在页眉的缩略词]{%大括号里面：输入文章标题}
	
	\author{Bang-Xian Han \thanks{School of Mathematics, Shandong University, 250100, Jinan, China, hanbx@sdu.edu.cn}
		\and Zhe-Feng Xu
		\thanks{School of Mathematical Sciences, University of Science and Technology of China,  230026, Hefei, China, xzf1998@mail.ustc.edu.cn}
		\thanks{SISSA,  34136, Trieste, Italy, zxu@sissa.it}
		\and Zhuo-Nan Zhu\thanks{School of Mathematical Sciences, University of Science and Technology of China, 230026, Hefei, China, zhuonanzhu@mail.ustc.edu.cn  }}
	
	\date{\today}

	\maketitle
	\begin{abstract}%摘要
		Numerous characterizations of Sobolev norms via the asymptotic behavior of non-local functionals have been established over the past decades; however, their validity beyond the PI framework remains poorly understood. We establish such a characterization on Lipschitz differentiability spaces without assuming either the doubling condition or a Poincar\'e inequality, by proving sharp two-sided Brezis--Van Schaftingen--Yung type asymptotic formulas. We also construct sharp counterexamples revealing the necessity of our assumptions,  and provide several examples which are of independent interest.
	\end{abstract}  
	\maketitle
	\MSC[2020]{30L15, 28A25, 46E35}
%	\tableofcontents

	%-------------------------------------------------------------------------
	
	\section{Introduction}
Beginning with the seminal Bourgain--Brezis--Mironescu (BBM) asymptotic formula \cite{MR3586796}, and culminating in the Brezis--Van Schaftingen--Yung (BVY) characterization \cite{MR4275122},  the characterizations of Sobolev norms via non-local functionals have reformed our understanding of weak differentiability.
However, a fundamental question remains unclear: 
\begin{center}
    \emph{To what extent do these characterizations via non-local functionals rely on the underlying metric measure geometry of the space?}
\end{center} All existing generalizations in the framework of general metric measure spaces,  for example \cite{MR4458224, MR3912793},  are confined to PI spaces, i.e., doubling metric measure spaces supporting a Poincar\'e inequality. This condition, while robust, excludes many geometrically relevant settings: totally disconnected Cantor sets with positive measure, ``rooms and passages" domains with severe bottlenecks, and certain fractals arising in geometric measure theory.

In this paper, we establish  a new characterization in the setting of Lipschitz differentiability spaces—a class introduced by Cheeger \cite{MR1708448} and refined by Keith \cite{MR2041901} and Bate \cite{zbMATH06394350} that precisely captures the validity of Rademacher's theorem in metric measure spaces. This class strictly contains PI spaces but also encompasses spaces lacking any quantitative connectivity, such as positive-measure Cantor sets in $\R^n$ and closure of disjoint union of Riemannian domains. This work resolves the question of whether non-local characterizations of Sobolev spaces persist in the absence of quantitative connectivity, and reveals a fundamental link between the BVY functional and the theory of Lipschitz differentiability spaces.

We work under the following assumptions:
\begin{assumption}\label{assump}
Let $(X,\d,\mm)$ be a metric measure space.  We assume
\begin{itemize}
\item [(A)]
Porous sets are $\mm$-negligible (a minimal regularity  satisfied by all Lipschitz differentiability spaces \cite{zbMATH06140723}). Here a Borel set $S\subset X$ is said to be porous if for every $x_0\in S$, there exist $\eta>0$ and a sequence ${\{x_m}\}_{m \in \N} \subseteq X $ with $ x_m \to x_0$ such that
\[
B_{\eta\d(x_m,x_0)}(x_m)\cap S=\emptyset,
\]
where $B_r(x):=\{y\in X: \d(x,y)<r\}$.
    \item [(B)]
    There exist constants $r_0, N>0$ such that the $N$-truncated maximal density function
\begin{equation*}\label{Theta}
\Theta_{N,r_0}(x):=\sup_{0<r\leq r_0}\frac{\mm 
\big(B_r(x)\big)}{r^N}
\end{equation*}
is integrable on bounded sets.  This condition serves as a sharp relaxation of upper Ahlfors regularity.
\end{itemize}
\end{assumption}

The sharpness of our assumptions is demonstrated by novel Counterexample \ref{counterex} showing that without the integrability condition (B), the Brezis--Van Schaftingen--Yung (BVY) characterization {\bf may  fail} even on Lipschitz differentiability spaces with positive density. This reveals a delicate interplay between the dimensional parameter $N$  and the local behavior of the measure.

Meanwhile,  there are many interesting nontrivial examples satisfying our assumptions.
\begin{example}
The following spaces satisfy Assumption \ref{assump} for suitable dimensional parameters $N$: 
    \begin{itemize}
     \item The space $(\Omega, \d_g\restr{\Omega}, \vol\restr{\Omega})$, where $\Omega \subseteq M^n$ is a closed subset with positive volume in an $n$-dimensional Riemannian manifold $(M^n,g)$;
        \item The space $(X,\d_g\restr{X}, \vol\restr{X})$, where $X=\overline{\bigcup_{i\in\N}\Omega_i}$ is the closure of a countable union of pairwise disjoint subsets $\{\Omega_i\}_{i\in\N}$ with $\vol(X)>0$ in a Riemannian manifold $(M^n,g)$;
\item The space $(K,|\cdot|_{{\mathrm{Euc}}}\restr{K},\mathcal{L}^n\restr{K})$, where $K\subseteq \R^n$ is a Cantor set with positive Lebesgue measure. Note that, being totally disconnected, this space fails to support any local Poincar\'e inequality or doubling condition;
\item The closure of the ``rooms and passages'' domain constructed in \cite{zbMATH00482856, zbMATH03606994}. Although this space is connected, its severe geometric bottlenecks cause it to fail both the local Poincar\'e inequality and the doubling condition;
\item The Laakso spaces introduced by Laakso in \cite{zbMATH01443217}, which are PI spaces; and the shortcut Laakso spaces recently introduced in \cite{arXiv:2510.25715}, which are purely PI-unrectifiable.
    \end{itemize}
\end{example}

%This property is naturally satisfied in the following two important cases:
%\begin{itemize}
%    \item $\mm$ satisfies the upper Ahlfors $N$-regularity assumption. There exists $C>0$ such that
%\[
%\mm(B_r(x))\leq C r^N
%\quad\text{for } x\in X,\ 0<r\leq r_0.
%\]
%In this case, $\Theta_{N,r_0}$ is actually $L^\infty(X,\mm)$-integrable.
%\item $\mm$ is a doubling measure with optimal doubling constant $\beta$, and $N= \frac{\log \beta}{\log 2}$, the doubling dimension (see Definition \ref{2.3}, Proposition \ref{2.166}).
%\end{itemize}

%for the upper Ahlfors $N$-regularity assumption, namely,
%\[
%\mm(B_r(x))\leq C r^N
%\quad\text{for } x\in X,\ 0<r\leq r_0,
%\]
%with some uniform constant $C$. In particular, for $\beta$-doubling metric measure spaces, the $L^1_{\mathrm{bloc}}(X,\mm)$-integrability of $\Theta_{N,r_0}$ can be replaced by the assumption $N\geq \frac{\log \beta}{\log 2}$ in our cases, where $\frac{\log \beta}{\log 2}$ is the doubling dimension.

Denote by ${\rm lip}(u)$ and ${\rm Lip}(u)$ the lower and upper pointwise Lipschitz constants, and by $\theta_N^-$ and $\theta_N^+$ the $N$-lower and $N$-upper pointwise density functions, respectively. We refer to Section~\ref{section2} for the remaining definitions. We now state our first main theorem concerning the BVY formula on metric measure spaces.

\begin{theorem}[Theorems \ref{T3.2} and \ref{T3.1}]
Let $p\geq 1$, $N>0$, and let $(X,\d,\mm)$ be a metric measure space. Under Assumption \ref{assump}, there exists $C:=C(p,N)$ such that for any $u\in {\rm Lip}_b(X,\d)$,
\begin{equation*}
    \begin{aligned}
        \varliminf_{\lambda\rightarrow +\infty} \lambda^p (\mm\times \mm)({E_{\lambda,u}})
        &\geq C\int_{ X}\theta_N^-(x)\,
\frac{\big({\rm{lip}}(u)(x)\big)^{N+p}}{\big({\rm{Lip}}(u)(x)\big)^N}\,\d\mm(x),\\
\varlimsup_{\lambda\rightarrow +\infty} \lambda^p (\mm\times \mm)({E_{\lambda,u}})
&\leq 2\int_{X}\theta_N^+(x)\big({\rm{Lip}}(u)(x)\big)^p\,\d\mm(x),
    \end{aligned}
\end{equation*}
where
\begin{equation}\label{E}
E_{\lambda,u}:=\left\{(x,y)\in X\times X: x\neq y,\, |u(x)-u(y)|\geq \lambda \big(\d(x,y)\big)^{\frac{N}{p}+1}\right\}.
\end{equation}
\end{theorem}

The proof requires new techniques compared to the PI setting. In the absence of Poincaré inequalities, we cannot control the oscillation of $u$  by its gradient. Instead, the lower bound (Theorem \ref{T3.2}) relies on a novel decomposition according to the ratio $\lip{u}/\Lip(u)$, while the upper bound (Theorem \ref{T3.1}) requires careful control of the maximal density function.

\medskip

\begin{remark}
In this theorem,  the integrability of $\Theta_{N,r_0}$ on bounded sets
is  only used for the upper bounds and cannot be omitted in general, even in the setting of Lipschitz differentiability spaces (see Counterexample~\ref{counterex}). This reveals a delicate interplay between the dimensional parameter $N$  and the infinitesimal structure of the space. Furthermore, this assumption can be replaced by $\Theta_{N,r_0}\in L^1_{\mathrm{loc}}(X,\mm)$ if one considers  $u\in {\rm Lip}_c(X,\d)$.
\end{remark}

\bigskip

In his seminal paper \cite{MR1708448}, Cheeger studied the differentiability of Lipschitz functions on  PI spaces.  He proved that every Lipschitz function $u$ is differentiable and that ${\rm Lip}(u)={\rm lip}(u)$ almost everywhere. Over the following decade, this ``generalized Rademacher theorem'' was further developed by Keith \cite{MR2041901} and several other authors, eventually leading to the notion  ``Lipschitz differentiability space" given by Bate \cite{zbMATH06394350}. A characterization \cite{zbMATH06394350, arXiv:1208.2869, MR2041901} due to Keith, Bate, and Gong asserts that $(X,\d,\mm)$ is a Lipschitz differentiability space if and only if ${\rm Lip}(u)$ and ${\rm lip}(u)$ are comparable for any Lipschitz function $u$ and the space is pointwise doubling. 

A deeper insight emerges when we focus on Lipschitz differentiability spaces. Recent breakthroughs by Bate, Eriksson-Bique, and Soultanis \cite{arXiv:2402.11284} establish that in such spaces, the minimal $*$-upper gradient $|\D u|_*$
   coincides with the pointwise Lipschitz constant ${\rm Lip}(u)$. By combining their results with our asymptotic formula, we obtain (Corollary \ref{C3.7}):
   \[
   \lim_{\lambda\to+\infty} \lambda^p (\mm\times\mm)(E_{\lambda,u}) \approx \||\D u|_*\|_{L^p}^p.
   \]
   
   {\bf To the best of our knowledge, this is the first non-local characterization of Sobolev norms on metric measure spaces that neither assumes the Poincaré inequality nor the doubling condition.} We  refer the reader to the  following papers studying the BVY formula from different perspectives:  one-parameter families of operators \cite{MR4512396}, ball Banach function spaces \cite{MR4525737,MR4645709,MR4777237}, general
domains \cite{poliakovsky2022some},  anisotropic spaces \cite{MR4557343},  Lorentz norms \cite{MR4279388}, Triebel--Lizorkin spaces \cite{MR4525722} and Orlicz spaces \cite{MR4549890,MR4732916}.

\begin{corollary}[Corollary \ref{C3.7}]\label{1.5}
Let $p,N\geq 1$, $b\geq a>0$, and let $(X,\d,\mm)$ be a Lipschitz differentiability space satisfying Assumption \ref{assump}(B). Assume that there exists $a\leq \theta^-_N(x)\leq \theta^+_N(x)\leq b$ for $\mm$-{\rm{a.e.}}\ $x\in X$. Then there exist constants $C_1:=C_1(p,N,a)$ and $C_2:=C_2(b)$ such that, for any $u\in {\rm Lip}_b(X,\d)$,
\begin{equation*}
\begin{aligned}
\varliminf_{\lambda\rightarrow +\infty} \lambda^p (\mm\times \mm)({E_{\lambda,u}})
&\geq C_1\||\D u|_*\|_{L^p(X,\mm)}^p,\\
\varlimsup_{\lambda\rightarrow +\infty} \lambda^p (\mm\times \mm)({E_{\lambda,u}})
&\leq C_2\||\D u|_*\|_{L^p(X,\mm)}^p.
\end{aligned}
\end{equation*}
\end{corollary}

\bigskip
Beyond the asymptotic formula itself, our results reveal {\bf structural rigidity} regarding the dimensional parameter $N$. In Theorem \ref{T3.5}, we prove that in Lipschitz differentiability spaces, the exponent $N\geq 1$  is uniquely determined: if the BVY formula yields finite non-zero limits for some Lipschitz function, then $N$  must equal the ``geometric dimension" of the space.
Furthermore, on doubling geodesic spaces, in terms of the doubling dimension, we establish an {\bf equivalence characterization} (Theorem \ref{final thm}): the BVY formula holds with two-sided bounds if and only if the $N$-lower and $N$-upper densities are uniformly comparable ($a_1 \leq \theta_N^- \leq \theta_N^+ \leq a_2$). This connects the analytic property (Sobolev characterization) directly with the geometric property (measure density). 
As for the case $0<N<1$, the BVY formula turns out to be degenerate in Lipschitz differentiability spaces (see Theorem \ref{degenerate}).

%Fortunately, our second main theorem partially answers this question: we establish an equivalence between the BVY formula and the density functions on doubling geodesic metric measure spaces. Note that we do not need integrability of $\Theta_{N,r_0}$ in the assumptions.
\begin{theorem}[Theorem \ref{final thm}]
Let $p\geq 1$ and let $(X,\d,\mm)$ be a geodesic metric measure space, where $\mm$ is a doubling measure with optimal doubling constant $\beta$. 
    %equipped with a $\beta$-doubling measure $\mm$. 
    Set $N=\frac{\log \beta}{\log 2}$ (the doubling dimension). Then the following statements are equivalent: 
		\begin{itemize}
		\item[\bf{(a)}] There exist positive constants $C_1$ and $C_2$ such that for any $u\in {\rm Lip}_b(X,\d)$,
		\begin{equation*}
			\begin{aligned}
				C_1\int_{X}\frac{\big({\rm lip}(u)(x)\big)^{N+p}}{\big({\rm Lip}(u)(x)\big)^N}&\,\d\mm(x) \leq \varliminf_{\lambda\rightarrow +\infty} \lambda^p (\mm\times \mm)({E_{\lambda,u}})\\
				 & \leq  \varlimsup_{\lambda\rightarrow +\infty} \lambda^p (\mm\times \mm)({E_{\lambda,u}}) \leq C_2\int_{X}\big({\rm Lip}(u)(x)\big)^p\,\d\mm(x),
			\end{aligned}
		\end{equation*}
		where $E_{\lambda,u}$ is defined in \eqref{E}.
		\item[\bf{(b)}] There exist positive constants $a_1$ and $a_2$ such that
        \begin{equation*}
            a_1\leq\theta_{N}^-(x)\leq \theta_{N}^+(x)\leq a_2, \quad\quad \mathrm{for}\,\,\mm\mathrm{\text{-}a.e. }\,x \in X.
        \end{equation*}
		\end{itemize}
\end{theorem}
% \begin{remark}
% We emphasize that Assumption \ref{assump}  is not required in this result. 
%     The uniqueness of the value of \(N\) in the above theorem stems from the following two facts: in the sense of Theorem \ref{T3.5}, finite nonvanishing density singles out \(N\) as the optimal exponent, while the doubling property forces this optimality to occur only when \(N=\frac{\log \beta}{\log 2}\).
% \end{remark}

The rest of this paper is organized as follows. In Section \ref{section2}, we recall necessary preliminaries. In Section \ref{section3}, we prove two-sided estimates for the BVY formula. In Section \ref{subsection}, we prove a characterization of Sobolev spaces on Lipschitz differentiability spaces. In Section \ref{subsections}, we prove equivalence results for these characterizations.
\bigskip

\noindent \textbf{Acknowledgements.}
We thank Yuanhe Wang, Liming Yin and Yuan Zhou for helpful discussions concerning Rademacher's theorem and Lipschitz differentiability  spaces. This work is supported by the Young Scientist Programs of the Ministry of Science \& Technology of China (2021YFA1000900, 2021YFA1002200) and by the NSFC grant 12201596. The second author is also supported by the China Scholarship Council (No.~202406340143).

\section{Preliminaries}\label{section2}
In this paper, $(X,\d)$ denotes a non-empty Polish metric space, and $\mm$ is a Radon measure such that $0<\mm(U)<+\infty $
for any non-empty bounded open set $U\subset X$. The triple $(X, \d,\mm)$ is called a metric measure space. We adopt the convention that $0\cdot\infty=0$ (and likewise $\infty\cdot 0=0$) whenever it appears.

\begin{definition}[Lipschitz constants]
Given a function $u:X\to\R$, the upper pointwise Lipschitz constant $\text{Lip}(u):X\to[0,+\infty]$ is defined by
\[
\text{Lip}(u)(x):=\varlimsup_{r\to 0^+}\sup_{y\in B_r(x)}\frac{|u(y)-u(x)|}{r}
=\varlimsup_{y\to x}\frac{|u(y)-u(x)|}{\d(x,y)},
\]
and the lower pointwise Lipschitz constant $\text{lip}(u):X\to[0,+\infty]$ is defined by
\[
\text{lip}(u)(x):=\varliminf_{r\to 0^+} \sup_{y\in B_r(x)}\frac{|u(y)-u(x)|}{r}.
\]
 The (global) Lipschitz constant is defined as
\[
{\bf Lip}(u):=\sup_{x\neq y}\frac{|u(y)-u(x)|}{\d(x,y)}.
\]
\end{definition}
If ${\bf Lip}(u)<+\infty$, we call $u$ a Lipschitz function and write $u\in \text{Lip}(X,\d)$.
We denote by $\mathrm{Lip}_b(X,\d)$ the collection of Lipschitz functions with bounded support, and by $\mathrm{Lip}_c(X,\d)$ the collection of Lipschitz functions with compact support. The following property provides an alternative definition of $\text{Lip}(u)(x)$ and $\text{lip}(u)(x)$.
	\begin{lemma}[\cite{MR2041901},  Lemma 4.1.2]\label{2.1}
		For any $u\in {\rm{Lip}}(X,\d)$ and $r>0$, define  $$l_ru(x)=\inf_{0<s\leq r}  \sup_{y\in B_s(x)}\frac{|u(y)-u(x)|}{s}\quad
		{\rm{and}} \quad
		L_ru(x)=\sup_{0<s\leq r}  \sup_{y\in B_s(x)}\frac{|u(y)-u(x)|}{s}.$$
		Then for any $x\in X$,
		$${\rm{lip}} (u)(x)=\lim_{r\rightarrow 0^+} {l_ru(x)}\,\quad {\rm{and}}\,\quad {\rm{Lip}} (u)(x)=\lim_{r\rightarrow 0^+} {L_ru(x)}.$$ In particular,  ${\rm{lip}} (u)$ and ${\rm{Lip}} (u)$ are Borel measurable.
	\end{lemma}

We also present a useful result concerning the upper pointwise Lipschitz constant, which will be used in Section \ref{subsection} and Section \ref{subsections}. We remark that we do not assume the set of isolated points has zero measure, which was assumed in \cite{zbMATH06394350}. 
\begin{lemma}[\cite{zbMATH06394350}, Corollary 4.9]\label{emptyset}
Let $(X,\d,\mathfrak{m})$ be a metric measure space. Suppose that for any $u\in\operatorname{Lip}_b(X,\d)$, it holds
\[
\operatorname{Lip}(u)(x)=0, \qquad \mathrm{for}\,\,\mm\mathrm{\text{-}a.e. }\,x \in X.
\]
Then $\mm(X\setminus {\rm{Iso}}(X))=0$, where ${\rm{Iso}}(X)$ denotes the set of isolated points in $X$. 
\end{lemma}

%which will be used in Corollary \ref{3.7} and Theorem \ref{T3.5}.
%\begin{lemma}[\cite{zbMATH06394350}, Corollary 4.9]\label{emptyset}
%Under our standing assumptions on $(X,\d,\mathfrak{m})$ (in particular, that $\mathfrak{m}$ has full support), $(X,\d,\mathfrak{m})$ cannot be a Lipschitz differentiability space satisfying the property that, for every $u\in\operatorname{Lip}_b(X,\d)$, one has
%\[
%\operatorname{Lip}(u)(x)=0 \qquad \mathrm{for}\,\,\mm\mathrm{\text{-}a.e. }\,x \in X.
%\]
%\end{lemma}

    \begin{definition}[Density functions]
For any $N>0$, the $N$-upper and $N$-lower pointwise density functions $\theta_N^\pm:X\to[0,+\infty]$ are defined by
\[
\theta_N^+(x):=\varlimsup_{r\to 0^+}\frac{\mm\big(B_r(x)\big)}{r^N}
\quad\text{and}\quad
\theta_N^-(x):=\varliminf_{r\to 0^+}\frac{\mm\big(B_r(x)\big)}{r^N}.
\]
\end{definition}

Next, we recall the definitions of doubling measure and doubling dimension, as well as their relation to porous sets, which will be used in Theorem \ref{final thm}. 
\begin{definition}[Doubling measure]\label{2.3}
We say that $\mm$ is a doubling measure if there exists $\beta_0>1$ such that
\[
\mm\big(B_{2r}(x)\big)\leq \beta_0\,\mm\big(B_r(x)\big), \quad \forall \, x\in X,\,\, r>0.
\]
In this case, we define the optimal doubling constant of $\mm$ by $$\beta:=\sup_{x\in X,\, r>0}\frac{\mm\big(B_{2r}(x)\big)}{\mm\big(B_r(x)\big)}\in (1,+\infty).$$
\end{definition} 

\begin{proposition}[\cite{zbMATH01474795}, Lemma 14.6;  \cite{zbMATH06387127}, Lemma 2.2]\label{2.166}
Let $\mm$ be a doubling measure with optimal doubling constant $\beta$.
%Let $\mm$ be a doubling measure with the optimal doubling constant $\beta$. 
Then for any $x_0\in X$, $r_0>0$, the following estimate holds:
\[
\frac{\mm\big(B_r(x)\big)}{\mm\big(B_{r_0}(x_0)\big)}\geq \frac{1}{\beta^2}\left(\frac{r}{r_0}\right)^{\frac{\log\beta}{\log2}},
\qquad \forall\, x\in B_{r_0}(x_0),\,\, r\in (0,r_0).
\]
In particular, the exponent $\frac{\log\beta}{\log2}$ is optimal (in terms of the doubling constant)
and  we call it the doubling dimension (cf. \cite[Section 3.4]{zbMATH06397370}, \cite[Proposition 3.14]{arXiv:2502.00825}).
\end{proposition}

By the Lebesgue differentiation theorem for doubling measures (cf. \cite[Section 3.4]{zbMATH06397370}), one can prove the following result.
	\begin{lemma}\label{2.11}
	Let $(X, \d, \mm)$ be a doubling metric measure space. Then every porous set in $X$ is $\mm$-null.
	\end{lemma}
On metric measure spaces, one can also study the differentiability of Lipschitz functions. We recall the following definitions introduced by Cheeger \cite{MR1708448}.

\begin{definition}
Let $n\in\N$ and let $(X,\d)$ be a metric space. A chart of dimension $n$ is a pair $(U,\phi)$, where $U\subset X$ is Borel and $\phi:X\to\R^n$ is Lipschitz. We say that a function $u:X\to\R$ is differentiable at $x_0\in U$ with respect to $(U,\phi)$ if there exists a unique $\d_{x_0}u\in \R^n$ such that
\[
\varlimsup_{x\to x_0}\frac{\left|u(x)-u(x_0)-\d_{x_0}u\cdot \big(\phi(x)-\phi(x_0)\big)\right|}{\d(x,x_0)}=0.
\]
\end{definition}

\begin{definition}[Lipschitz differentiability space]
We say that a metric measure space $(X,\d,\mm)$ is a Lipschitz differentiability space if there exists a countable decomposition of $X$ into charts such that every $u\in {\rm Lip}(X,\d)$ is differentiable at $\mm$-a.e.\ point of each chart.
\end{definition}

\begin{theorem}[\cite{zbMATH06140723}, Theorem 2.4]\label{porous null}
Let $(X,\d,\mm)$ be a Lipschitz differentiability space. Then every porous set in $X$ is $\mm$-null.
\end{theorem}

Cheeger's seminal paper \cite{MR1708448} established that PI spaces are Lipschitz differentiability spaces. Moreover, Lipschitz differentiability spaces can be characterized by the following ${\rm Lip}$-${\rm lip}$ condition, introduced by Keith \cite{MR2041901}.

\begin{definition}
We say that a metric measure space $(X,\d,\mm)$ satisfies the ${\rm Lip}$-${\rm lip}$ condition if there exist a countable Borel decomposition $X=\bigcup_i X_i$ and, for each $i\in\N$, a constant $\delta_i>0$ such that, for any $u\in {\rm Lip}(X,\d)$,
\[
{\rm Lip}(u)(x)\leq \delta_i\,{\rm lip}(u)(x),\quad \text{for } \mm\text{-a.e.\ } x\in X_i.
\]
\end{definition}

\begin{theorem}[\cite{zbMATH06394350, arXiv:1208.2869, MR2041901}]\label{2.8}
Let $(X,\d,\mm)$ be a metric measure space. Then $(X,\d,\mm)$ is a Lipschitz differentiability space if and only if it satisfies the ${\rm Lip}$-${\rm lip}$ condition and is pointwise doubling (that is, the doubling constant depends on the point; cf.\ \cite[Definition 8.2]{zbMATH06394350}).
\end{theorem}

Remarkably, the inequality in the ${\rm Lip}$-${\rm lip}$ condition can be strengthened to an equality. As mentioned above, for any $u\in {\rm Lip}(X,\d)$, there exists a minimal $*$-upper gradient in such a space, denoted by $|\D u|_*$. We refer to \cite[Section~2]{arXiv:2402.11284} for details.

\begin{theorem}[\cite{arXiv:2402.11284, zbMATH06559620}]\label{2.14}
$(X,\d,\mm)$ is a Lipschitz differentiability space if and only if, for any $u\in {\rm Lip}(X,\d)$,
\[
{\rm Lip}(u)(x)= {\rm lip}(u)(x)=|\D u|_*(x), \qquad \mathrm{for}\,\,\mm\mathrm{\text{-}a.e. }\,x \in X.
\]
\end{theorem}

\begin{remark}
There are several notions of weak gradient in the setting of metric measure spaces. In general, one cannot replace $|\D u|_*$ in Theorem~\ref{2.14} by Cheeger's minimal $p$-weak upper gradient $|\D u|_p$. A counterexample is the Lipschitz differentiability space $(K,|\cdot|_{\text{Euc}}\restr{K},\mathcal{L}^n\restr{K})$, where $K\subseteq \R^n$ is a Cantor set with positive Lebesgue measure. In such a space, $|\D u|_p=0$ for any $u\in {\rm Lip}_b(K)$ and $p\geq 1$. We refer to \cite{zbMATH06216025,zbMATH07423256} for further discussions.
\end{remark}

There is another, more geometric characterization of Lipschitz differentiability spaces, arising from geometric measure theory, specifically through Alberti representations. We refer to \cite{zbMATH00487038} and \cite[Section 2]{zbMATH06394350} for precise definitions and detailed properties of Alberti representations. The following structural theorem will be used in the proof of Theorem \ref{degenerate}.

\begin{theorem}[\cite{zbMATH06394350}, Theorem 7.8]\label{Alberti}
$(X,\d,\mm)$ is a Lipschitz differentiability space if and only if there exists a countable Borel decomposition $X = \cup_i U_i$ such that each $\mm \restr U_i$ has a finite universal collection of Alberti representations.
\end{theorem}

% \begin{lemma}\label{2.11}
% Let $(X,\d,\mm)$ be a metric measure space. Assume that
% \begin{itemize}
% \item[\bf a)] $\mm$ is doubling;
% \item[\bf b)] every porous set in $X$ is $\mm$-null;
% \item[\bf c)] $\mm$ is pointwise doubling.
% \end{itemize}
% Then the implications $\bf a) \Rightarrow \bf b)\Rightarrow \bf c)$ hold.
% \end{lemma}
% 	\begin{proof}
% 		${\bf{b})\Rightarrow \bf{c})}$ is a consequence of \cite[Lemma 8.3]{B15}.
		
% 		${\bf {a}) \Rightarrow \bf{b})}$ Denote $\beta$ the doubling constant. For any porous set $S\subseteq X$ and $x\in S$, there exists $\eta>0$ and $X\ni y_n\rightarrow x$ such that
% 		\begin{equation}
% 			B_{\eta r_n}(y_n)\subseteq B_{(1+\eta)r_n}(x), \quad 	B_{\eta r_n}(y_n)\cap S=\emptyset,
% 		\end{equation}
% 	    where $r_n:=\d(x,y_n)$. Then by Lemma \ref{2.166}, we have
% 	    \begin{equation}
% 	    	\frac{\mm(B_{(1+\eta) r_n}(x)\setminus S)}{\mm(B_{(1+\eta)r_n}(x))}\geq \frac{\mm(B_{\eta r_n}(y_n))}{\mm(B_{(1+\eta)r_n}(x))}\geq \frac{1}{\beta^2}\left(\frac{\eta}{1+\eta}\right)^{\frac{\log \beta}{\log 2}}>0.
% 	    \end{equation}
% 	    Letting $r_n\rightarrow0$, we have
% 	    \begin{equation}
% 	    	\varliminf_{r_n\rightarrow 0}\frac{\mm(B_{(1+\eta) r_n}(x)\cap S)}{\mm(B_{(1+\eta)r_n}(x))}\leq 1-\frac{1}{\beta^2}\left(\frac{\eta}{1+\eta}\right)^{\frac{\log \beta}{\log 2}}<1,
% 	    \end{equation}
% 	    which implies $x$ is not a Lebesgue point. By Lebesgue differentiation theorem, we then obtain $\mm(S)=0$.
% 	\end{proof}

\section{Two-sided estimates for the BVY formula}\label{section3}
	\subsection{A lower bound estimate}\label{subsection3.2}
\begin{theorem}\label{T3.2}
Let $p\geq 1$, $N>0$, and let $(X,\d,\mm)$ be a metric measure space. Suppose that Assumption~\ref{assump}(A) holds. Then there exists a constant $C:=C(p,N)>0$ such that, for any $u\in \mathrm{Lip}_b(X,\d)$,
\begin{equation}\label{main}
\varliminf_{\lambda\rightarrow +\infty} \lambda^p (\mm\times \mm)({E_{\lambda,u}}) \geq C\int_{X}\theta_N^-(x)\,
\frac{\big({\rm{lip}}(u)(x)\big)^{N+p}}{\big({\rm{Lip}}(u)(x)\big)^N}\,\d\mm(x),
\end{equation}
where $E_{\lambda,u}$ is defined in \eqref{E}.
\end{theorem}
\begin{proof}
%It suffices to prove \eqref{main} on the set $\{x\in X:\theta_N^-(x)>0\}$.
%Fix $u\in {\rm{Lip}}_b(X,\d)$. Then $\mm(\text{supp}\,u)<+\infty$. 
For $x\in X$, denote
\begin{equation}\label{3.2}
\hat{E}_{\lambda,u}(x):=\left\{y\in X: y\neq x,\ |u(x)-u(y)|\geq \lambda \big(\d(x,y)\big)^{\frac{N}{p}+1}\right\}.
\end{equation}
%We only consider the set 
%\begin{equation}\label{SS}
%S:=\big\{x\in X: {\rm{Lip}}(u)(x)>0\big\} \subseteq \mathrm{supp}\,u.
%\end{equation}
%Otherwise, if ${\rm{Lip}}(u)(x)=0$, then
%\begin{equation}\label{Lipzero}
%\hat{E}_{\lambda,u}(x)=\emptyset.
%\end{equation}

For any $u\in {\rm Lip}_b(X,\d)$, there exists a bounded set $S\subseteq X$ with $\mm(S)<+\infty$ such that $\supp(u)\subseteq S$. By the Fubini--Tonelli theorem,
\begin{equation}\label{33}
(\mm\times \mm)({E_{\lambda,u}}) \geq \int_{S}\mm\big(\hat{E}_{\lambda,u}(x)\big)\,\d\mm(x).
\end{equation}

Denote
\begin{equation}
\theta_{N,M}^-(x):=\min\big\{\theta_N^-(x),M\big\}, \quad M\in \N,
\end{equation}
and
\begin{equation}\label{key}
L_i:=\left\{x\in  S: \frac{1}{2^{i+1}}{{\rm{Lip}}(u)}(x)< {\rm{lip}}(u)(x)\leq \frac{1}{2^i}{\rm{Lip}}(u)(x)\right\}, \quad i\in \N.
\end{equation}
In particular, each $L_i$ has finite measure and
\begin{equation}\label{Bigl}
\lim_{M\rightarrow+\infty} \theta_{N,M}^-(x)=\theta_{N}^-(x),\quad
S = \bigcup_{i\in \N} L_i \cup \big\{x\in S: {\rm{lip}}(u)(x)=0\big\}.
\end{equation}
Note that it is sufficient to consider points in $\{L_i\}_{i\in \N}$.

We divide the proof into five steps.

\paragraph{Step 1.}
Let $i_0\in\N$ be such that $\mm(L_{i_0})>0$. Given $0<\epsilon<\frac{1}{4}\mm(L_{i_0})$ satisfying
\begin{equation}\label{3.5}
\mm\big\{x\in L_{i_0}: {\rm{lip}} (u)(x)>2^{i_0+3}\epsilon\big\}>\frac{2}{3}\mm(L_{i_0}).
\end{equation}
{\bf Claim:} There exist $K_\epsilon \subseteq L_{i_0}$ and $r_{\epsilon}>0$ such that:
\begin{itemize}
\item it holds that
\begin{equation}\label{3.6}
\mm(L_{i_0}\setminus K_\epsilon)<\epsilon;
\end{equation}

\item for any $0<r\leq r_{\epsilon}$ and $x\in K_\epsilon$,
\begin{equation}\label{3.77}
\sup_{y\in B_r(x)}\frac{|u(y)-u(x)|}{r}< 2^{{i_0}+1}{\rm{lip}} (u)(x)+\epsilon;
\end{equation}

\item for any $x,y\in K_\epsilon$ with $\d(x,y)\leq r_\epsilon$,
\begin{equation}\label{3.8}
\left|{\rm{lip}} (u)(x)-{\rm{lip}} (u)(y)\right|< \epsilon,\quad |\theta_{N,M}^-(x)-\theta_{N,M}^-(y)|< \epsilon;
\end{equation}

\item for any $x\in K_\epsilon$ and $0<r\leq r_{\epsilon}$,
\begin{equation}\label{3.9}
\mm\big(B_r(x)\big)\geq \big(\theta_{N,M}^-(x)-\epsilon\big)r^N.
\end{equation}
\end{itemize}

Firstly, by Lemma \ref{2.1} 
%${\rm{Lip}} (u)(x)=\lim_{r\rightarrow 0^+} {L_ru}(x)$ pointwise. By
and Egoroff's theorem, there exists $K^1_{\epsilon}\subseteq L_{i_0}$ with $\mm\left(L_{i_0}\setminus K^1_{\epsilon}\right)<\frac{\epsilon}{4}$ such that $L_ru$ converges to ${\rm{Lip}} (u)$ uniformly on $K^1_{\epsilon}$. Hence, there exists $r^1_{\epsilon}>0$ such that for any $0<r\leq r^1_{\epsilon}$ and $x\in K^1_{\epsilon}$, it holds
\[
L_ru(x)=\sup_{0<s\leq r} \sup_{y\in B_s(x)}\frac{|u(y)-u(x)|}{s}\leq {\rm{Lip}} (u)(x)+\epsilon,
\]
so that
\[
\sup_{y\in B_r(x)}\frac{|u(y)-u(x)|}{r}\leq {\rm{Lip}} (u)(x)+\epsilon< 2^{{i_0}+1}{\rm{lip}} (u)(x)+\epsilon.
\]

Next, by Lusin's theorem, there exists a closed set $K^2_{\epsilon}\subseteq L_{i_0}$ such that $\mm(L_{i_0}\setminus K^2_{\epsilon})<\frac{\epsilon}{4}$ and ${\rm{lip}} (u)$,  $\theta_{N,M}^-$ are continuous on $K^2_{\epsilon}$. Moreover, there exists a compact set $K^3_\epsilon\subseteq K^2_\epsilon$ such that $\mm(K^2_\epsilon\setminus K^3_\epsilon)<\frac{\epsilon}{4}$. By the Heine--Cantor theorem, there exists $r^2_{\epsilon}>0$ such that for any $x,y\in K^3_\epsilon$ with $\d(x,y)\leq r^2_{\epsilon}$, the inequalities in \eqref{3.8} hold.

Finally, denote
\begin{equation*}
\begin{aligned}
G_k:=\left\{x\in L_{i_0}:
\forall\, r\in \mathbb{Q}\cap \left(0,\frac{1}{k}\right], \,\mm\big(B_r(x)\big)\geq \big(\theta_{N,M}^-(x)-\epsilon\big)r^N\right\},\quad k\in\N^+.
\end{aligned}
\end{equation*}
Then $G_k$ is Borel and $\mm\left(L_{i_0}\setminus\bigcup_k G_k\right)=0$. Thus, there exist $K_\epsilon^4\subseteq L_{i_0}$ and $r_\epsilon^3>0$ such that $\mm(L_{i_0}\setminus K^4_{\epsilon})<\frac{\epsilon}{4}$ and for any $r\in \mathbb{Q}\cap (0,r_\epsilon^3]$ and $x\in K_\epsilon^4$, it holds
\begin{equation*}
\mm\big(B_r(x)\big)\geq \big(\theta_{N,M}^-(x)-\epsilon\big)r^N.
\end{equation*}
By a standard density argument, we obtain that \eqref{3.9} holds for any $0<r\leq r_\epsilon^3$.

Set $K_\epsilon:= K^1_{\epsilon}\cap K^2_{\epsilon}\cap K^3_{\epsilon}\cap K^4_{\epsilon}$ and $r_\epsilon:=\min\{r^1_{\epsilon}, r^2_{\epsilon}, r^3_{\epsilon}\}>0$. Then the pair $(K_\epsilon, r_\epsilon)$ satisfies the claim.

       \paragraph{Step 2.}
Define
\begin{equation*}
\begin{aligned}
    {\rm lip}_{K_\epsilon}(u)(x):&=\varliminf_{r\to 0^+}\sup_{y\in B_r(x)\cap K_\epsilon}\frac{|u(y)-u(x)|}{r},\\
    l_r^{K_\epsilon} u(x):&=\inf_{0<s\leq r}\sup_{y\in B_s(x)\cap K_\epsilon}\frac{|u(y)-u(x)|}{s}.
\end{aligned}
\end{equation*}
It is obvious that ${\rm lip}_{K_\epsilon}(u)(x)\leq {\rm lip}(u)(x)$ for any $x\in K_\epsilon$. Moreover, by the same argument as in Lemma \ref{2.1}, we have that $l_r^{K_\epsilon} u$ converges to ${\rm lip}_{K_\epsilon}(u)$ for any $x\in K_\epsilon$. Denote
\begin{equation}
U:=\left\{x\in K_\epsilon:\ {\rm lip}_{K_\epsilon}(u)(x)< \frac{3}{4}{\rm lip}(u)(x)\right\}.
\end{equation}
\noindent\textbf{Claim:}  $\mm(U)=0$.

%Fix $x\in U$. By the definitions of $U$ and ${\rm{lip}}_{K_\epsilon}(u)(x)$, 

For any $x\in U$, there exists a sequence $\{r_n\}_{n\in\N}$ converging to $0$ such that
\begin{equation}\label{3.11}
\sup_{y\in B_{r_n}(x)\cap K_\epsilon}\frac{|u(y)-u(x)|}{r_n}\leq \frac{7}{8}{\rm lip}(u)(x).
\end{equation}
Let $\eta:= \frac{{\rm lip}(u)(x)}{16({\rm lip}(u)(x)+ {\bf Lip}(u))}>0.$
%\[
%\eta:= \frac{{\rm lip}(u)(x)}{16({\rm lip}(u)(x)+ {\bf Lip}(u))}>0.
%\]
There exists $y_n\in B_{\frac{r_n}{1+\eta}}(x)$ such that
\begin{equation}
\frac{|u(y_n)-u(x)|}{r_n}\geq \frac{15}{16(1+\eta)}{\rm lip}(u)(x).
\end{equation}
Consider the metric ball $B_{\eta\d(y_n,x)}(y_n)$. Note that, for any $w\in B_{\eta\d(y_n,x)}(y_n)$,
\[
\d(w,x)\leq \d(w,y_n)+\d(y_n,x)\leq (1+\eta)\d(y_n,x)\leq r_n,
\]
which implies that $B_{\eta\d(x,y_n)}(y_n)\subseteq B_{r_n}(x).$
%\[
%B_{\eta\d(x,y_n)}(y_n)\subseteq B_{r_n}(x).
%\]

Assume for contradiction that there exists $z\in B_{\eta\d(y_n,x)}(y_n)\cap U \subseteq B_{r_n}(x)$. Since
\[
|u(z)-u(y_n)|\leq {\bf Lip}(u)\d(z,y_n)\leq {\bf Lip}(u)\eta\d(y_n,x)\leq {\bf Lip}(u)\frac{\eta r_n}{1+\eta},
\]
we obtain
\begin{equation}
\begin{aligned}
\frac{|u(x)-u(z)|}{r_n}
&\geq \frac{|u(y_n)-u(x)|-|u(z)-u(y_n)|}{r_n}\\
&\geq \frac{15}{16(1+\eta)}{\rm lip}(u)(x)- {\bf Lip}(u)\frac{\eta}{1+\eta}> \frac{7}{8}{\rm lip}(u)(x),
\end{aligned}
\end{equation}
which contradicts \eqref{3.11}. Hence $B_{\eta\d(x,y_n)}(y_n)\cap U=\emptyset,$ and then $U$ is porous at $x$.
%\[
%B_{\eta\d(x,y_n)}(y_n)\cap U=\emptyset.
%\]
By the arbitrariness of $x$, we conclude that $U$ is a porous set; thus $\mm(U)=0$. Consequently, for $\mm$-a.e.\ $x\in K_\epsilon$, it holds that
\begin{equation}\label{compare}
{\rm lip}_{K_\epsilon}(u)(x)\geq \frac{3}{4}{\rm lip}(u)(x)>0.
\end{equation}

By Egoroff's theorem, there exists a set $K'_\epsilon\subseteq K_\epsilon\setminus U$ with $\mm(K_\epsilon\setminus K'_\epsilon)<\epsilon$ such that $l_r^{K_\epsilon} u$
%\begin{equation}\label{lrk}
%l_r^{K_\epsilon} u(x):=\inf_{0<s\leq r}\sup_{y\in B_s(x)\cap K_\epsilon}\frac{|u(y)-u(x)|}{s}
%\end{equation}
converges to ${\rm lip}_{K_\epsilon}(u)$ uniformly on $K'_\epsilon$. Therefore, there exists $0<r'_\epsilon\leq r_\epsilon$ such that for any $0<r\leq r'_\epsilon$ and $x\in K'_\epsilon$, it holds
\begin{equation}\label{uniform}
l_r^{K_\epsilon} u(x)\geq {\rm lip}_{K_\epsilon}(u)(x)-\epsilon\geq \frac{3}{4}{\rm lip}(u)(x)-\epsilon.
\end{equation}

	\paragraph{Step 3.} Define
\begin{equation}\label{3.14}
A_{\epsilon, i_0}:=\big\{x\in K'_\epsilon \subseteq K_\epsilon\subseteq L_{i_0}: {\rm lip}(u)(x)>2^{i_0+3}\epsilon\big\},
\end{equation}
then $A_{\epsilon, i_0}$ is nonempty. Indeed, if $A_{\epsilon, i_0}=\emptyset$, then by the choice of $\epsilon$, $K'_\epsilon$, and \eqref{3.6},
\[
\mm\big(\{x\in K'_\epsilon \subseteq K_\epsilon \subseteq L_{i_0}: {\rm lip}(u)(x)\leq 2^{i_0+3}\epsilon\}\big)
=\mm(K'_\epsilon)\geq \mm(L_{i_0})-2\epsilon>\frac{1}{2}\mm(L_{i_0}),
\]
which contradicts \eqref{3.5}.

Fix $x\in A_{\epsilon,i_0}$ and $0<r\leq r'_\epsilon$. By \eqref{uniform} and \eqref{3.14}, there exists $y_{x,r}\in B_{\frac{r}{2}}(x)\cap K_\epsilon$ such that
\begin{equation}\label{320}
\frac{|u(y_{x,r})-u(x)|}{\frac{r}{2}} \geq \sup_{y\in B_{\frac{r}{2}}(x)\cap K_\epsilon}\frac{|u(y)-u(x)|}{\frac{r}{2}}-\epsilon \geq \frac{1}{2}{\rm lip}(u)(x).
\end{equation}
%\begin{equation}
%\begin{aligned}
%\label{diyi}
%\frac{|u(y_{x,r})-u(x)|}{\frac{r}{2}} &\geq  \sup_{y\in B_{\frac{r}{2}}(x)\cap K_\epsilon}\frac{|u(y)-u(x)|}{\frac{r}{2}}-\frac{1}{2^{i_0+3}}{\rm{lip}} (u)(x)\\
%&\geq~ l_r^{K_\epsilon} u(x)-\epsilon-\frac{1}{2^{i_0+3}}{\rm{lip}} (u)(x)\\
%&> ~\frac{3}{4}{\rm{lip}} (u)(x)-\frac{1}{2^{i_0+3}}{\rm{lip}} (u)(x)-\frac{1}{2^{i_0+3}}{\rm{lip}} (u)(x)\\
%& \geq~\frac{1}{2}{\rm{lip}} (u)(x).
%\end{aligned}
%\end{equation}
Moreover, we have
\begin{equation}\label{g5.5}
\begin{aligned}
\sup_{z\in B_{\frac{r}{2^{i_0+5}}}(y_{x,r})}|u(z)-u(y_{x,r})|
&\overset{\eqref{3.77}}{\leq} \frac{r}{2^{i_0+5}}\big(2^{i_0+1}{\rm lip}(u)(y_{x,r})+\epsilon\big)\\
&\overset{\eqref{3.8}}{\leq} \frac{r}{2^{i_0+5}}\big(2^{i_0+1}{\rm lip}(u)(x)+2^{i_0+1}\epsilon+\epsilon\big)\\
&\overset{\eqref{3.14}}{\leq} \frac{r}{8}{\rm lip}(u)(x).
\end{aligned}
\end{equation}
%\begin{equation}\label{g5.5}
%\begin{aligned}
%\sup_{z\in B_{\frac{r}{2^{i_0+5}}}(y_{x,r})}|u(z)-u(y_{x,r})|
%&\leq~\frac{r}{2^{i_0+5}}\big(2^{{i_0}+1}{\rm{lip}} (u)(y_{x,r})+\epsilon\big)\\
%&\leq~\frac{r}{2^{i_0+5}}\big(2^{{i_0}+1}{\rm{lip}} (u)(x)+2^{{i_0}+1}\epsilon+\epsilon\big)\\
%&<~\frac{r}{8}{\rm{lip}} (u)(x).
%\end{aligned}
%\end{equation}
%In particular, combining the above two inequalities, we have
%\begin{equation}\label{xneqz}
%  x\notin B_{\frac{r}{2^{i_0+5}}}(y_{x,r}).
%\end{equation}

      \paragraph{Step 4.}
Fix $x\in A_{\epsilon, i_0}$. Let $\lambda, r>0$ be such that $\lambda r^{\frac{N}{p}}= \frac{1}{8}\lip u(x)$.
%\begin{equation}\label{lamda}
%    \lambda r^{\frac{N}{p}}= \frac{1}{8}\lip u(x).
%\end{equation}
Note that $0<r\leq r'_\epsilon$ holds provided that $\lambda=\lambda(\epsilon)$ is sufficiently large. Denote
\[
S_{i_0}(x,r):=\left\{z\in B_r(x): z\neq x, \,|u(z)-u(x)|\geq \frac{1}{8}{\rm{lip}} (u)(x)\d(z,x)\right\}.
\]
%For any $z\in B_{\frac{r}{2^{i_0+5}}}(y_{x,r})$, we have 
%\begin{equation}\label{leqr}
%    0<\d(z,x)\leq \d(z,y_{x,r})+\d(y_{x,r},x)\leq r.
%\end{equation}
For any $z\in B_{\frac{r}{2^{i_0+5}}}(y_{x,r})$, by \eqref{320} and \eqref{g5.5}, we have
\begin{equation}\label{big0}
|u(x)-u(z)|
\geq |u(x)-u(y_{x,r})|-|u(z)-u(y_{x,r})|
\geq\frac{1}{8}{\rm{lip}} (u)(x)\d(z,x),
\end{equation}
which implies that
\begin{equation}
B_{\frac{r}{2^{i_0+5}}}(y_{x,r})\subseteq S_{i_0}(x,r).
\end{equation}
%Moreover, for any $z \in S_{i_0}(x,r)$, by \eqref{lamda} and %\eqref{leqr}, we can see that
%\begin{equation*}
%|u(z)-u(x)|
%\geq \frac{1}{8}{\rm{lip}} (u)(x)\d(z,x)
%= \frac{1}{8}\lambda r^{\frac{N}{p}}\d(z,x)
%\geq \frac{1}{8}\lambda (\d(z,x))^{\frac{N}{p}+1},
%\end{equation*}
%which implies
Moreover, recalling the definition of $\hat{E}_{\lambda,u}(x)$ in \eqref{3.2}, we can see that
\begin{equation}\label{3.19}
B_{\frac{r}{2^{i_0+5}}}(y_{x,r})
\subseteq S_{i_0}(x,r)
\subseteq \hat E_{\lambda, u}(x).
\end{equation}

Since $y_{x,r}\in B_{\frac{r}{2}}(x)\cap K_\epsilon$, combining this with \eqref{3.19}, \eqref{3.8} and \eqref{3.9}, we obtain
\begin{equation}\label{222}
\mm\big(\hat E_{\lambda, u}(x)\big)
\geq \mm\left(B_{\frac{r}{2^{i_0+5}}}(y_{x,r})\right)
\geq \big(\theta_{N,M}^-(x)-2\epsilon\big)\left(\frac{r}{2^{i_0+5}}\right)^N.
\end{equation}
Therefore,
\begin{equation*}
\varliminf_{\lambda\rightarrow +\infty}
\lambda^p \int_{L_{i_0}}\mm\big(\hat E_{\lambda, u}(x)\big)\,\d \mm(x)
\geq \frac{1}{2^{(i_0+5)N}8^{p}}
\int_{A_{\epsilon,i_0}}
\big(\theta_{N,M}^-(x)-2\epsilon\big) |\lip u(x) |^p\,\d\mm(x).
\end{equation*}
Letting $\epsilon\rightarrow 0$ and then $M\rightarrow +\infty$, we obtain
\begin{equation}\label{final}
\varliminf_{\lambda\rightarrow +\infty}
\lambda^p \int_{L_{i_0}}\mm\big(\hat E_{\lambda, u}(x)\big)\, \d \mm(x)
\geq \frac{1}{2^{(i_0+5)N}8^{p}}
\int_{L_{i_0}}
\theta_N^-(x) |\lip u(x) |^p\d\mm(x).
\end{equation}

\paragraph{Step 5.}
Finally, by Fatou's lemma, we have  
\begin{eqnarray*}
&&\varliminf_{\lambda\rightarrow +\infty} \lambda^p (\mm\times \mm)({E_{\lambda,u}})\\
&\overset{\eqref{33}}{\geq}&
\varliminf_{\lambda\rightarrow +\infty}
\lambda^p \int_{S}\mm\big( \hat{E}_{\lambda,u}(x)\big)\,\d\mm(x)\\
&\overset{{\rm{Fatou}}}{\geq}&
\sum_{i=0}^{+\infty}
\varliminf_{\lambda\rightarrow +\infty}
\lambda^p \int_{L_i} \mm\big( \hat{E}_{\lambda,u}(x)\big)\,\d\mm(x)
\\
&\overset{\eqref{final}}{\geq}&
\sum_{i=0}^{+\infty}
\int_{L_{i}}
\frac{1}{2^{(i+5)N}8^{p}}
\theta_N^-(x) |{\rm{lip}}(u)(x)|^p \,\d\mm(x)
\\
&\stackrel{\eqref{key}}{\geq}&
\frac{1}{2^{5N+3p}}
\int_{S}
\theta_N^-(x)
\frac{\big({\rm{lip}}(u)(x)\big)^{N+p}}{\big({\rm{Lip}}(u)(x)\big)^N}\,\d\mm(x)\\
&=&
\frac{1}{2^{5N+3p}}
\int_{X}
\theta_N^-(x)
\frac{\big({\rm{lip}}(u)(x)\big)^{N+p}}{\big({\rm{Lip}}(u)(x)\big)^N}\,\d\mm(x),
\end{eqnarray*}
which completes the proof.
\end{proof}

\subsection{An upper bound estimate}\label{subsection3.1}
\begin{theorem}\label{T3.1}
Let $p\geq 1$, $N>0$, and let $(X,\d,\mm)$ be a metric measure space. Suppose that Assumption~\ref{assump}(B) holds. Then, for any $u\in \mathrm{Lip}_b(X,\d)$,
\begin{equation*}
\varlimsup_{\lambda\rightarrow +\infty} \lambda^p (\mm\times \mm)({E_{\lambda,u}})\leq 2\int_{X}\theta_N^+(x)\big({\rm{Lip}}(u)(x)\big)^p\,\d\mm(x),
\end{equation*}
where $E_{\lambda,u}$ is defined in \eqref{E}.
\end{theorem}
\begin{proof}
For any $u\in {\rm Lip}_b(X,\d)$, there exists a bounded set $S\subseteq X$ with $\mm(S)<+\infty$ such that $\supp(u)\subseteq S$. Recalling the definition of $\hat{E}_{\lambda,u}(x)$ in \eqref{3.2}, the Fubini--Tonelli theorem yields
%Fix $u\in {\rm{Lip}}_b(X,\d)$. For any $x\in X$, denote
%\begin{equation}\label{3.22}
%\hat{E}_{\lambda,u}(x):=\{y\in X: y\neq x, |u(x)-u(y)|\geq \lambda \big(\d(x,y)\big)^{\frac{N}{p}+1}\}.
%\end{equation}
%By the Fubini theorem, we have
\begin{equation}\label{3.24}
(\mm\times \mm)({E_{\lambda,u}})\leq 2\int_{S}\mm\big( \hat{E}_{\lambda,u}(x)\big)\,\d\mm(x).
\end{equation}

Fix $x \in S$. For any $\epsilon>0$, there exists $\delta:=\delta(x,\epsilon)>0$ such that, for any $y \in B_\delta(x)$,
$$\frac{|u(x)-u(y)|}{\d(x,y)}\leq \text{Lip} (u)(x)+\epsilon.$$
Let $\lambda: =\lambda(x, \epsilon)$ be sufficiently large so that
\begin{equation}\label{leqD}
\left(\frac{\text {Lip}(u)(x)+\epsilon}{\lambda}\right)^{\frac{p}{N}}\leq \left(\frac{\text {\bf Lip}(u)+\epsilon}{\lambda}\right)^{\frac{p}{N}}<\delta.
\end{equation}
Thus, for any $y\in \hat{E}_{\lambda,u}(x)$, we have
\begin{equation*}
\lambda \delta^{\frac{N}{p}}>\text{\bf Lip} (u)+\epsilon>\text{\bf Lip} (u)\geq \frac{|u(x)-u(y)|}{\d(x,y)} \geq\lambda \big(\d(x,y)\big)^\frac{N}{p}.
\end{equation*}
Hence $\hat{E}_{\lambda,u}(x)\subseteq B_\delta(x)$, and then
\begin{equation}\label{3.25}
\begin{aligned}
\mm\big(\hat{E}_{\lambda,u}(x)\big)
=~ \mm\big(\hat{E}_{\lambda,u}(x)\cap B_\delta(x)\big) 
\le ~\mm\Big(
B_{\left(\frac{\operatorname{Lip}(u)(x)+\varepsilon}{\lambda}\right)^{\frac{p}{N}}}(x)
\Big).
\end{aligned}
\end{equation}
For any $\lambda$ satisfying $0<\left(\frac{ {\bf Lip}(u)+\epsilon}{\lambda}\right)^{\frac{p}{N}}\leq \min \{r_0, \delta\}$, by Assumption~\ref{assump}(B), we have
\begin{equation*}\label{weak}
\lambda^p \mm\big(\hat E_{\lambda, u}(x)\big)\leq \frac{\mm\left(B_{\left(\frac{ {\bf Lip}(u)+\epsilon}{\lambda}\right)^{\frac{p}{N}}}(x)\right)}{\lambda^{-p}}\leq \Theta_{N,r_0}(x)\big( {\bf Lip}(u)+\epsilon\big)^p\in L^1(S,\mm).
\end{equation*}
Moreover, combining \eqref{3.25} and letting $\lambda\rightarrow+\infty$, we have
\begin{equation}\label{3.27}
    \varlimsup_{\lambda\rightarrow +\infty}\lambda^p \mm\big(\hat E_{\lambda, u}(x)\big)\leq\theta_N^+(x)\left({\rm Lip}(u)(x)+\epsilon\right)^p.
\end{equation}
Therefore, by Fatou's lemma, we obtain
\begin{align*}
    \varlimsup_{\lambda\rightarrow +\infty} \lambda^p (\mm\times \mm)({E_{\lambda,u}})
    &\overset{\eqref{3.24}}{\leq} 2\varlimsup_{\lambda\rightarrow +\infty} \lambda^p \int_{S}\mm\big(\hat E_{\lambda, u}(x)\big)\,\d\mm(x)\\
    &\overset{{\rm Fatou}}{\leq} 2 \int_{S}\varlimsup_{\lambda\rightarrow +\infty} \lambda^p\mm\big(\hat E_{\lambda, u}(x)\big)\,\d\mm(x)\\
    &\overset{\eqref{3.27}}{\leq} 2 \int_{X}\theta_N^+(x)\big({\rm Lip}(u)(x)+\epsilon\big)^p \,\d\mm(x).
\end{align*}
Letting $\epsilon\rightarrow0$, we obtain the result.
%Letting $\epsilon\rightarrow 0$, we obtain
%\begin{equation}
%\varlimsup_{\lambda\rightarrow +\infty} \lambda^p (\mm\times \mm)({E_{\lambda,u}})\leq 2 \int_{X}\theta_N^+(x)\big({\rm{Lip}}(u)(x)\big)^p \,\d\mm(x),
%\end{equation}
%which is the thesis.
\end{proof}

  The following counterexample demonstrates that the Assumption \ref{assump}(B) in Theorem~\ref{T3.1} is indispensable in general. Since its construction is technical and auxiliary to the main text, we postpone the detailed proof to Appendix~\ref{appendix}.

\begin{counterexample}\label{counterex}
Let $p,N\geq 1$. There exists a Lipschitz differentiability space $(X,\d,\mm)$ satisfying $\theta_N^+(x)>0$ for $\mm$-{\rm{a.e.}}\ $x\in X$, and a Lipschitz function $u\in\Lip_b(X,\d)$, such that
\begin{equation*}
\varlimsup_{\lambda\rightarrow +\infty} \lambda^p (\mm\times \mm)({E_{\lambda,u}})=+\infty, \quad\mathrm{while }\quad \int_{X}\theta_N^+(x)\big({\rm Lip}(u)(x)\big)^p\,\d\mm(x)<+\infty,
\end{equation*}
where $E_{\lambda,u}$ is defined in \eqref{E}.
\end{counterexample}

   \begin{remark}
       For $0<N<1$, consider the base space $(\R^1, |\cdot|_{\mathrm{Euc}}, \mathcal{L}^1)$. By applying the same construction, we have   $\theta_N^+(x)=0$ for $\mm$-a.e.\ $x\in \R$, and hence the right-hand side vanishes. However, the upper limit on the left-hand side still diverges to infinity. Therefore, the conclusion of Counterexample~\ref{counterex} remains valid.
   \end{remark}

\section{Lipschitz differentiability spaces}\label{subsection}
In this section, we restrict our attention to Lipschitz differentiability spaces. We first consider the case $0<N<1$, in which the BVY formula turns out to be degenerate.

\begin{theorem}\label{degenerate}
   Let $p\geq 1$, $0<N<1$, and let $(X,\d,\mm)$ be a Lipschitz differentiability space. Suppose that Assumption~\ref{assump}(B) holds. Then, for any $u\in \mathrm{Lip}_b(X,\d)$,
    \begin{equation*}
        \lim_{\lambda\rightarrow +\infty} \lambda^p (\mm\times \mm)({E_{\lambda,u}})=0.
    \end{equation*}
\end{theorem}

\begin{proof}
    Denote
    \begin{equation*}
        F:=\{x\in X: \theta_N^+(x)>0\},\quad F_k:=\{x\in X: \theta_N^+(x)\geq \frac{1}{k}\},\quad k\in \N.
    \end{equation*}
{\bf Claim:}  $\mathcal{H}^1(F)=0$, where $\mathcal{H}^1$ denotes the $1$-dimensional Hausdorff measure.

Fix $k\in \N$ and a bounded subset $K\subseteq X$. For any $\delta\in (0,1)$ and any $x\in K\cap F_k$, there exists $r_x\in (0,\frac{\delta}{10})$ such that
\begin{equation}\label{theta}
    \mm\big(B_{r_x}(x)\big)\geq \frac{1}{2k}r_x^N.
\end{equation}
Consider the covering $\mathcal{F}=\{B_{r_x}(x)\}_{x\in K\cap F_k}$ of $K\cap F_k$. By the $5B$-covering lemma \cite[Section 3.3]{{zbMATH06397370}}, there exists a countable sequence $\{B_{r_i}(x_i)\}_{i\in\N}$ of disjoint balls in $\mathcal{F}$, such that
\begin{equation}
    K\cap F_k\subseteq \bigcup_{i\in\N}B_{5r_i}(x_i)\subseteq \{x\in X: \d(x, K)\leq 1\}.
\end{equation}
Thus, by \eqref{theta}, we have  
\begin{equation}
\begin{aligned}
     \sum_{i\in \N} \diam \big(B_{5r_i}(x_i)\big)^N
     \leq 10^N2k\sum_{i\in \N}\mm\big(B_{r_i}(x_i)\big)\leq 10^N2k\mm\big(\{x\in X: \d(x, K)\leq 1\}\big).
\end{aligned}
\end{equation}
%\begin{equation}
%\begin{aligned}
%     \sum_{i\in \N} \diam (B_{5r_i}(x_i))^N
%     \leq  &\sum_{i\in \N} (10r_i)^N\\
%     \leq &10^N2k\sum_{i\in \N}\mm(B_{r_i}(x_i))\leq 10^N2k\mm\left(\{x\in X: \d(x, K)\leq 1\}\right).
%\end{aligned}
%\end{equation}
Letting $\delta\rightarrow0$, we get
\begin{equation}
    \mathcal{H}^N(K\cap F_k)\leq 10^N2k\mm\big(\{x\in X: \d(x, K)\leq 1\}\big)<+\infty,
\end{equation}
which implies $\mathcal{H}^1(K\cap F_k)=0$ since $0<N<1$. By the arbitrariness of $K$, it holds that $\mathcal{H}^1(F_k)=0$. Since $F=\bigcup_{k\in\N}F_k$, we have $\mathcal{H}^1(F)=0$.

 By Theorem \ref{Alberti}, there exists a countable Borel decomposition $X = \cup_j U_j$ such that each $\mm \restr U_j$ has a finite universal collection of Alberti representations. Note that  Assumption~\ref{assump}(B) implies every isolated point is $\mm$-null. Combined with the separability of $X$, this yields $\mm(\mathrm{Iso}(X))=0$; hence, every $0$-dimensional chart is necessarily $\mm$-null (cf. \cite[Remark 3.2]{zbMATH06394350}, Lemma \ref{emptyset}).  
 For $j\in \N^+$, there is an Alberti representation $(\mathbb{P},\{\mm_\gamma\})$ of $\mm \restr U_j$ (cf. \cite[Section 2]{zbMATH06394350}), such that
\begin{equation}
    \mm(F\cap U_j)=\int_{\Gamma(X)}\mm_{\gamma}(F\cap U_j)\,\d\mathbb{P}(\gamma),\quad \mm_{\gamma}\ll \mathcal{H}^1\restr {\mathrm{Im} \gamma}.
\end{equation}
Note that 
\begin{equation}
    \mathcal{H}^1(F\cap U_j \cap  {\mathrm{Im} \gamma})\leq \mathcal{H}^1(F)=0 \quad \Longrightarrow\quad \mm_{\gamma}(F\cap U_j)=0,
\end{equation}
which implies $\mm(F\cap U_j)=0$ for any $j\in \N$, and hence $\mm(F)=0$. (One can also deduce $\mm(F)=0$ by combining the Claim above with \cite[Theorem 4.1.1]{zbMATH07517884}, which yields a more direct argument.) Thus $\theta_N^+(x)=0$ for $\mm$-a.e. $x\in X$; combining this with Theorem \ref{T3.1}, we complete the proof.
\end{proof}

In contrast to the degenerate case $0<N<1$, Counterexample \ref{counterex} demonstrates that for $N \geq 1$, the $N$-upper density function $\theta_N^+$ does not necessarily vanish. Consequently, by combining Theorem \ref{T3.2}, Theorem \ref{T3.1}, Theorem \ref{porous null} and Theorem \ref{2.14}, we obtain the following characterization of Sobolev spaces on Lipschitz differentiability spaces.
\begin{corollary}\label{C3.7}
Let $p,N\geq 1$, $b\geq a>0$, and let $(X,\d,\mm)$ be a Lipschitz differentiability space satisfying Assumption \ref{assump}(B). Assume that there exists $a\leq \theta^-_N(x)\leq \theta^+_N(x)\leq b$ for $\mm$-{\rm{a.e.}}\ $x\in X$. Then there exist constants $C_1:=C_1(p,N,a)$ and $C_2:=C_2(b)$ such that, for any $u\in {\rm Lip}_b(X,\d)$,
\begin{equation*}
\begin{aligned}
&\varliminf_{\lambda\rightarrow +\infty} \lambda^p (\mm\times \mm)({E_{\lambda,u}})\geq C_1\||\D u|_*\|_{L^p(X,\mm)}^p,\\
&\varlimsup_{\lambda\rightarrow +\infty} \lambda^p (\mm\times \mm)({E_{\lambda,u}})\leq C_2\||\D u|_*\|_{L^p(X,\mm)}^p.
\end{aligned}
\end{equation*}
\end{corollary}

 Compared with Theorem \ref{degenerate}, the BVY formula does not degenerate when $N \geq 1$. Consequently, by Theorem \ref{T3.2}, Lemma \ref{emptyset}, Theorem \ref{2.14} and the assumption $\supp(\mm)=X$, we have the following corollary. 

 %\begin{corollary}\label{3.7}
    %Let $p,N\geq 1$, and let $(X,\d,\mm)$ be a Lipschitz differentiability space such that
	%$\theta^-_N(x)>0$ for $\mm$-a.e.\ $x\in X$. Suppose that for any %$u\in {\rm{Lip}}_b(X,\d)$, 
    %\begin{equation*}
     %   \lim_{\lambda\rightarrow +\infty} \lambda^p (\mm\times \mm)({E_{\lambda,u}})=0.
    %\end{equation*}
    %Then the following are equivalent:
    %	\begin{itemize}
	%	\item[\bf(a)] For every $u\in \mathrm{Lip}_b(X,\d)$,
	%	\[
	%	\lim_{\lambda\to+\infty}\lambda^p(\mm\times\mm)(E_{\lambda,u})=0.
	%	\]
	%	\item[\bf(b)] $(X,\d,\mm)$ is purely $0$-dimensional as a Lipschitz differentiability space,
	%	i.e.\ it admits no chart of positive dimension on a set of positive $\mm$-measure.
	%\end{itemize}
    %Moreover, these conditions imply $\mm$ is supported on a purely $1$-unrectifiable set.
%\end{corollary}

\begin{corollary}\label{3.7}
    Let $p,N\geq 1$. There does not exist a Lipschitz differentiability space $(X,\d,\mm)$ with $0<\theta^-_N(x)\leq \theta^+_N(x)<+\infty$ for $\mm\text{-\rm{a.e.} }x\in X$, such that for any $u\in {\rm{Lip}}_b(X,\d)$, 
    \begin{equation*}
        \lim_{\lambda\rightarrow +\infty} \lambda^p (\mm\times \mm)({E_{\lambda,u}})=0.
    \end{equation*}
\end{corollary}

\section{Equivalence of the characterizations}\label{subsections}
In this final section, we present two equivalent characterizations bridging the BVY formula and the local dimensional and density structures of the underlying metric measure space. Our first result concerns the uniqueness of the dimensional parameter $N$ in Lipschitz differentiability spaces.
\begin{theorem}\label{T3.5}
Let $p,N,\bar{N}\geq 1$, and let $(X,\d,\mm)$ be a Lipschitz differentiability space. Suppose that Assumption~\ref{assump}(B) holds. If $\theta_N^-(x)>0$ for $\mm\text{-\rm{a.e.} }x\in X$, then the following statements are equivalent:
\begin{itemize}
\item[\bf{(a)}] There exists $u\in {\rm Lip}_b(X,\d)$ such that
\begin{equation}\label{3.0}
0<\varliminf_{\lambda\rightarrow +\infty} \lambda^p (\mm\times \mm)({\bar{E}_{\lambda,u}})\leq
\varlimsup_{\lambda\rightarrow +\infty} \lambda^p (\mm\times \mm)({\bar{E}_{\lambda,u}})<+\infty,
\end{equation}
where
\begin{equation}\label{que}
\bar{E}_{\lambda,u}:=\left\{(x,y)\in X\times X: x\neq y,\ |u(x)-u(y)|\geq \lambda \big(\d(x,y)\big)^{\frac{\bar{N}}{p}+1}\right\};
\end{equation}

\item[\bf{(b)}] $\bar{N}=N$.
\end{itemize}
\end{theorem}
\begin{proof}
${\bf (b)} \Rightarrow \bf (a)$. Let $\bar{N}=N$. For $u\in\Lip(X,\d)$, by Theorem \ref{2.14}, we have  $\Lip(u)=\text{lip}(u)$ for $\mm$-a.e. $x\in X$.
%Note that Assumption~\ref{assump}(B) implies $\theta_N^+$ is integrable on bounded sets. Combined with the separability of $X$, this yields $\theta^+_N(x)<+\infty$ for $\mm\text{-\rm{a.e.} }x\in X$. 
%Note that if $\Theta_{N,r_0}$ is integrable on bounded sets, then so is $\theta_N^+$. 
%Therefore, Theorem~\ref{T3.1} yields the desired finite upper bound and Corollary \ref{3.7} yields the strictly positive lower bound. 
Similar to Theorem \ref{degenerate}, we have $\mm(\mathrm{Iso}(X))=0$. Combining with Lemma \ref{emptyset} and the assumption $\supp(\mm)=X$, there exists $u\in {\rm Lip}_b(X,\d)$ such that $$\mm\big(\{x\in X: {\rm Lip}(u)(x)>0\}\big)>0.$$
Combining this with the assumption on $\theta_N^-(x)$ and applying Theorem \ref{T3.2} implies the strictly positive lower bound.
%   \begin{equation*}
%\varliminf_{\lambda\rightarrow +\infty} \lambda^p (\mm\times \mm)({\bar E_{\lambda,u}}) \geq C\int_{X}\theta_N^-(x)\,
%\frac{\big({\rm{lip}}(u)(x)\big)^{N+p}}{\big({\rm{Lip}}(u)(x)\big)^N}\,\d\mm(x)>0.
%\end{equation*}

On the other hand, note that if $\Theta_{N,r_0}$ is integrable on bounded sets, then so is $\theta_N^+$. Therefore, Theorem~\ref{T3.1} yields the desired finite upper bound.
%\begin{equation*}
%\varlimsup_{\lambda\rightarrow +\infty} 
%\lambda^p (\mm\times \mm)({\bar E_{\lambda,u}})\leq 
%2\big(\text {\bf Lip} (u)\big)^p\|\theta^+_N\|_{L^1_{{\rm{bloc}}}(X,\mm)}<+\infty.
%\end{equation*}

${\bf (a)} \Rightarrow \bf (b)$. We proceed by contradiction.
\newline
{\bf Case 1:} $\bar{N}<N$. Similar to the proof of Theorem \ref{T3.1}, we have  
	\begin{equation*}
		\begin{aligned}
			\varlimsup_{\lambda\rightarrow +\infty} \lambda^p (\mm\times \mm)({\bar{E}_{\lambda,u}})
				\overset{\eqref{3.24}}{\leq}& 2\varlimsup_{\lambda\rightarrow +\infty} \lambda^p \int_{S}\mm\big( \hat{\bar{E}}_{\lambda,u}(x)\big)\,\d\mm(x)\\
			\overset{\eqref{3.25}}{\leq}& 2 \varlimsup_{\lambda\rightarrow +\infty} \lambda^p\int_{S} \mm\left(B_{\left(\frac{\text {Lip}(u)(x)+\epsilon}{\lambda}\right)^\frac{p}{\bar{N}}}(x)\right) \d \mm(x)\\
			\overset{\phantom{\rm{Fatou}}}{\leq}& 2\varlimsup_{\lambda\rightarrow +\infty}\lambda^{p\left(1-\frac{N}{\bar{N}}\right)}\int_{S} \lambda^{\frac{pN}{\bar{N}}}\mm\left(B_{\left(\frac{\text {Lip}(u)(x)+\epsilon}{\lambda}\right)^\frac{p}{\bar{N}}}(x)\right) \d \mm(x)\overset{*}{=}0,
		\end{aligned}
	\end{equation*}
	where $(*)$ follows from the integrability of $\Theta_{N,r_0}$ on bounded sets. This contradicts the strictly positive lower bound in \eqref{3.0}.
    \newline
{\bf Case 2:} $\bar{N}>N$. By Theorem \ref{T3.1} and the strictly positive lower bound in \eqref{3.0}, we have $$\mm\big(\{x\in X: \text{lip}(u)(x)>0\}\big)=\mm\big(\{x\in X: \Lip(u)(x)>0\}\big)>0.$$ 

We adopt the same notation and argument as in the proof of Theorem \ref{T3.2}. Fix $i_0\in\N$ with $\mm(L_{i_0})>0$. We choose $\epsilon>0$ such that $A_{\epsilon, i_0}\subseteq L_{i_0}$, defined by \eqref{3.14}, satisfies $\mm(A_{\epsilon,i_0})>0$ and $$\lip u(x)=\Lip(u)(x)>2^{i_0+3}\epsilon, \qquad\forall \, x \in  A_{\epsilon,i_0}.$$

Since $\theta_N^-(x)>0$ for $\mm$-a.e.\ $x\in X$, by shrinking $\epsilon>0$ if necessary, there exists $M\in \N$ such that $\mm(A_{\epsilon,i_0}\cap \{\theta_{N,M}^-\geq 3\epsilon\})>0$. Denote
\[
A:=A_{\epsilon, i_0}\cap \{\theta_{N,M}^-\geq 3\epsilon\}.
\]
Then, for $x\in A$, letting $\lambda, r>0$ be such that $\lambda r^{\bar{N}/p}=\frac{1}{8}\lip u(x)$, we have
\begin{equation*}
				\begin{aligned}
					\varliminf_{\lambda\rightarrow +\infty} \lambda^p (\mm\times \mm)({\bar{E}_{\lambda,u}})
						\overset{\eqref{33}}{\geq}& \varliminf_{\lambda\rightarrow +\infty} \lambda^p \int_{L_{i_0}}\mm\big(\hat {\bar{E}}_{\lambda, u}(x)\big)\, \d \mm(x)\\
						\overset{\eqref{222}}{\geq}& \varliminf_{\lambda\rightarrow +\infty}\lambda^p\int_{A} (\theta_{N,M}^-(x)-2\epsilon)\left(\frac{r}{2^{i_0+5}}\right)^N \,\d\mm(x)\\
					\overset{\phantom{\rm{Fatou}}}{\geq}& \varliminf_{\lambda\rightarrow +\infty} \frac{\lambda^{p\left(1-\frac{N}{\bar{N}}\right)}}{ 2^{(i_0+5)N}8^{\frac{pN}{\bar{N}}} } \int_{A}\epsilon |\lip u(x)|^{\frac{pN}{\bar{N}}}\,\d\mm(x)=+\infty,
				\end{aligned}
		\end{equation*}
which contradicts the finite upper bound in \eqref{3.0}.

Therefore, combining {\bf{Case 1}} and {\bf{Case 2}}, we conclude that $\bar{N}=N$.
\end{proof}

As our second result, we provide an equivalent characterization of the BVY formula in terms of the density functions on doubling geodesic metric measure spaces.
	%We recall the result from \cite[Lemma 8.3]{zbMATH06394350} first.
	%\begin{proposition}
	%Let $(X, \d, \mm)$ be a pointwise doubling metric measure space. Then there exists a Borel decomposition 
	%\begin{equation}
	%	X=N\cup\bigcup_{k=1}^{+\infty} X_k,
	%\end{equation}
	%where $\mm(N)=0$ and for $k\in\N$, $(X_k, \d\restr{k}, \mm\restr{X_k})$ is a doubling metric measure space with doubling constant $\beta_k$.
	%\end{proposition}

\begin{theorem}\label{final thm}
	Let $p\geq 1$ and let $(X,\d,\mm)$ be a geodesic metric measure space, where $\mm$ is a doubling measure with optimal doubling constant $\beta$. 
    %equipped with a $\beta$-doubling measure $\mm$. 
    Set $N=\frac{\log \beta}{\log 2}$ (the doubling dimension). Then the following statements are equivalent: 
		\begin{itemize}
		\item[\bf{(a)}] There exist positive constants $C_1$ and $C_2$ such that for any $u\in {\rm Lip}_b(X,\d)$,
		\begin{equation*}
			\begin{aligned}
				C_1\int_{X}\frac{\big({\rm lip}(u)(x)\big)^{N+p}}{\big({\rm Lip}(u)(x)\big)^N}&\,\d\mm(x) \leq \varliminf_{\lambda\rightarrow +\infty} \lambda^p (\mm\times \mm)({E_{\lambda,u}})\\
				 & \leq  \varlimsup_{\lambda\rightarrow +\infty} \lambda^p (\mm\times \mm)({E_{\lambda,u}}) \leq C_2\int_{X}\big({\rm Lip}(u)(x)\big)^p\,\d\mm(x),
			\end{aligned}
		\end{equation*}
		where $E_{\lambda,u}$ is defined in \eqref{E}.
		\item[\bf{(b)}] There exist positive constants $a_1$ and $a_2$ such that
        \begin{equation*}
            a_1\leq\theta_{N}^-(x)\leq \theta_{N}^+(x)\leq a_2, \quad\quad \mathrm{for}\,\,\mm\mathrm{\text{-}a.e. }\,x \in X.
        \end{equation*}
		\end{itemize}
\end{theorem}

\begin{proof}
${\bf (b)} \Rightarrow {\bf (a)}$. Since $\mm$ is a doubling measure, combining Lemma \ref{2.11} with Theorem \ref{T3.2} directly yields the lower bound, where $C_1:=C_1(a_1,p,N)$.

As for the upper bound, similar to the proof of Theorem \ref{T3.1}, we have  
\begin{equation}\label{thees}
		\varlimsup_{\lambda\rightarrow +\infty} \lambda^p (\mm\times \mm)({E_{\lambda,u}})
	\leq 2\varlimsup_{\lambda\rightarrow +\infty} \lambda^p \int_{S}\mm\big( \hat{E}_{\lambda,u}(x)\big)\,\d\mm(x).
\end{equation}
Fix $\epsilon>0$. By Proposition \ref{2.166}, for any $0<r<\left(\frac{ {\bf Lip}(u)+\epsilon}{\lambda}\right)^{\frac{p}{N}}$, we have
\begin{equation*}
	\frac{\mm\left(B_{\left(\frac{ {\bf Lip}(u)+\epsilon}{\lambda}\right)^{\frac{p}{N}}}(x)\right)}{\lambda^{-p}}\leq \beta^2\left( \frac{\left(\frac{ {\bf Lip}(u)+\epsilon}{\lambda}\right)^{\frac{p}{N}}}{r}\right)^{\frac{\log \beta}{\log 2}}\frac{\mm\big(B_r(x)\big)}{\lambda^{-p}}=\beta^2\big( {\bf Lip}(u)+\epsilon\big)^p\frac{\mm\big(B_r(x)\big)}{r^N}.
\end{equation*}
Letting $r\rightarrow 0^+$, we obtain
\begin{equation*}
	\lambda^p\mm\left(B_{\left(\frac{ {\rm Lip}(u)(x)+\epsilon}{\lambda}\right)^{\frac{p}{N}}}(x)\right)\leq a_2 \beta^2\big( {\bf Lip}(u)+\epsilon\big)^p\in L^1(S,\mm).
\end{equation*}
Thus, by \eqref{thees}, \eqref{3.25} and \eqref{3.27}, we obtain the upper bound, where $C_2:=C_2(a_2)$.
%\begin{equation*}
%\begin{aligned}
%\varlimsup_{\lambda\rightarrow +\infty} \lambda^p (\mm\times \mm)({E_{\lambda,u}})
%&\leq~2 \int_{\mathrm{supp}\,u}\varlimsup_{\lambda\rightarrow +\infty} \lambda^p\mm\big( \hat{E}_{\lambda,u}(x)\big)\,\d\mm(x)\\
%&\leq ~2 \int_{\mathrm{supp}\,u}\varlimsup_{\lambda\rightarrow +\infty} \lambda^p\mm\left(B_{\left(\frac{ {\rm Lip}(u)(x)+\epsilon}{\lambda}\right)^{\frac{p}{N}}}(x)\right) \,\d\mm(x)\\
%&\leq ~C_2 \int_{X}\big({\rm{Lip}}(u)(x)+\epsilon\big)^p \,\d\mm(x),
%\end{aligned}
%\end{equation*}

${\bf (a)} \Rightarrow {\bf (b)}$.  For any $x_0\in X$, since $\mm$ is a Radon measure, we have $\mm(\partial B_R(x_0))=0$ for a.e.\ $R>0$.
For such $x_0\in X$ and $R>0$ with $\mm(\partial B_R(x_0))=0$, define $u(x):=\max\big\{0, R-\d(x,x_0)\big\}\in {\rm Lip}_b(X,\d)$.
% then $\text{\bf Lip} (u)=1$. 
Since $(X,\d)$ is a geodesic space, for $\mm$-a.e.\ $x\in X$, we have
\begin{equation*}
	\text{Lip} (u)(x)=\text{lip} (u)(x)=1,\,\, \text{on } B_R(x_0),\quad 	\text{Lip} (u)(x)=\text{lip} (u)(x)=0,\,\,\text{on } X\setminus  B_R(x_0).
\end{equation*}
Therefore, by assumption {\bf (a)}, we have
\begin{equation*}
	C_1\mm\big(B_R(x_0)\big)\leq \varliminf_{\lambda\rightarrow +\infty} \lambda^p (\mm\times \mm)({E_{\lambda,u}})
	 \leq  \varlimsup_{\lambda\rightarrow +\infty} \lambda^p (\mm\times \mm)({E_{\lambda,u}}) \leq	C_2\mm\big(B_R(x_0)\big).
\end{equation*}
Then there exists $\lambda_0:=\lambda_0(x_0,R)$ such that for any $\lambda\geq \lambda_0$, we have
\begin{equation}\label{3.35}
		\frac{C_1}{2}\mm\big(B_R(x_0)\big)\leq \lambda^p (\mm\times \mm)({E_{\lambda,u}})\leq 2C_2\mm\big(B_R(x_0)\big).
\end{equation}

Denote $r:=\lambda^{-\frac{p}{N}}$ and choose $0<t\leq \frac{2}{3}\left(\frac{1}{3}\right)^{\frac{p}{N}}$. Recalling the definition of $\hat{E}_{\lambda,u}(x)$ in \eqref{3.2}, for any $x\in B_R(x_0)$, if $y\in  \hat{E}_{\lambda,u}(x)$, then
\begin{equation}
	\lambda\big(\d(x,y)\big)^\frac{N}{p}\leq \frac{|u(x)-u(y)|}{\d(x,y)}\leq 1,
\end{equation}
which implies $\hat{E}_{\lambda,u}(x)\subseteq B_{r}(x)$.
%\begin{equation}\label{da}
%    \hat{E}_{\lambda,u}(x)\subseteq B_{r}(x).
%\end{equation}

\noindent\textbf{Claim:} For any $x\in B_{R-\frac{tr}{2}}(x_0)\setminus B_{tr}(x_0)$, there exists $y\in B_R(x_0)$ such that $B_{\frac{tr}{2}}(y)\subseteq \hat{E}_{\lambda,u}(x)$.
%\begin{equation}\label{xiao}
%    B_{\frac{tr}{2}}(y)\subseteq \hat{E}_{\lambda,u}(x).
%\end{equation}

Fix $x\in B_{R-\frac{tr}{2}}(x_0)\setminus B_{tr}(x_0)$, and let $\gamma(s):[0,\d(x,x_0)]\rightarrow X$ be a unit-speed geodesic
from $x$ to $x_0$. Denote $y=\gamma(tr)$. Then $y\in B_R(x_0)$ and $u(y)=R-\d(y,x_0)$, hence
	\begin{equation}\label{eq:uy-ux}
	\big|u(y)-u(x)\big|=\left|\big(R-\d(y,x_0)\big)-\big(R-\d(x,x_0)\big)\right|=tr.
\end{equation}

For any $z\in B_{\frac{tr}{2}}(y)$, we have $z\in B_R(x_0)$ and
\begin{equation}\label{3.37}
	\d(z,x)\leq \d(z,y)+\d(y,x)\leq \frac{3tr}{2}.
\end{equation}
Moreover, since $u$ is $1$-Lipschitz, we have
\begin{equation}
	|u(x)-u(z)|\geq |u(x)-u(y)|-|u(y)-u(z)|\geq tr-\d(y,z)\geq \frac{tr}{2}.
\end{equation}
Therefore, by the choice of $r$ and $t$, we obtain
\begin{equation}
		|u(x)-u(z)|\geq \frac{r}{2}\geq \left(\frac{3t}{2}\right)^{\frac{N}{p}+1}r\overset{\eqref{3.37}}{\geq}\lambda \big(\d(x,z)\big)^{\frac{N}{p}+1},
\end{equation}
which implies that $z\in \hat{E}_{\lambda,u}(x)$. This proves the claim.

Since, by \eqref{3.37}, we have $B_{\frac{tr}{2}}(y)\subseteq B_{\frac{3tr}{2}}(x)$, Proposition \ref{2.166} yields
\begin{equation}\label{xiao2}
	\frac{\mm\big(B_{\frac{tr}{2}}(y)\big)}{\mm\big(B_{\frac{3tr}{2}}(x)\big)}\geq \frac{1}{\beta^2}\left(\frac{1}{3}\right)^
	\frac{\log \beta}{\log 2}.
\end{equation}
Thus, by \eqref{33} and \eqref{3.24}, we have
\begin{equation}\label{3.42}
\begin{aligned}
    C(\beta,p,N)&\int_{B_{R-\frac{tr}{2}}(x_0)\setminus B_{tr}(x_0)}\frac{\mm\big(B_{\frac{3tr}{2}}(x)\big)}{\left(\frac{3tr}{2}\right)^N}\,\d\mm (x)\\
    &\leq \lambda^p (\mm\times \mm)({E_{\lambda,u}})\leq 2\int_{B_R(x_0)}\frac{\mm\big(B_r(x)\big)}{r^N}\,\d\mm(x).
\end{aligned}
\end{equation}

Combining \eqref{3.35} with \eqref{3.42}, and choosing $r_0:=r_0(x_0,R)=\min\{ \lambda_0^{-p/N}, R/t\}$, for any $0<r\leq r_0$, we have  
\begin{equation}\label{low}
	\begin{aligned}
			\int_{B_{\frac{R}{2}}(x_0)\setminus B_{tr}(x_0)}\frac{\mm\big(B_{\frac{3tr}{2}}(x)\big)}{\left(\frac{3tr}{2}\right)^N}\,\d\mm(x)	
	\leq C(\beta,p,N,C_2)\mm\big(B_R(x_0)\big),
	\end{aligned}
\end{equation}
%\begin{equation}\label{low}
%	\begin{aligned}
%			\int_{B_{\frac{R}{2}}(x_0)\setminus B_{tr}(x_0)}\frac{\mm\left(B_{\frac{3}{2}tr}(x)\right)}{\left(\frac{3}{2}tr\right)^N}\,\d\mm(x)
%			&\leq  \int_{B_{R-\frac{t}{2}r}(x_0)\setminus B_{tr}(x_0)}\frac{\mm\left(B_{\frac{3}{2}tr}(x)\right)}{\left(\frac{3}{2}tr\right)^N}\,\d\mm(x)\\
%	&\leq C(\beta,p,N,C_2)\mm(B_R(x_0)),
%	\end{aligned}
%\end{equation}
and
\begin{equation}\label{upp}
	 \int_{B_R(x_0)}\frac{\mm\big(B_r(x)\big)}{r^N}\,\d\mm(x)\geq C(C_1)\mm\big(B_R(x_0)\big).
\end{equation}

Note that
\begin{equation}
	\frac{\mm\big(B_r(x)\big)}{r^N}\leq \beta\frac{\mm\big(B_\frac{r}{2}(x)\big)}{r^N}= \frac{\mm\big(B_\frac{r}{2}(x)\big)}{\left(\frac{r}{2}\right)^N},
\end{equation}
Then denote $r_n:=2^{-n}$ and
\begin{equation}
	f(x):=\lim_{n\rightarrow +\infty}\frac{\mm\big(B_{r_n}(x)\big)}{r_n^N},\quad g(x):=\lim_{n\rightarrow +\infty}\frac{\mm\big(B_{\frac{3tr_n}{2}}(x)\big)}{\left(\frac{3tr_n}{2}\right)^N}.
\end{equation}
Recall that $t\leq \frac{2}{3}\left(\frac{1}{3}\right)^{\frac{p}{N}}<\frac{2}{3}$, and hence $B_{\frac{3tr_n}{2}}(x)\subset B_{r_n}(x)$. By Proposition \ref{2.166}, we have $f(x)\leq C(\beta, p, N)g(x)$. Let $n$ be sufficiently large and replace $r$ with $r_n$ in \eqref{low}. Note that for $m\leq n$, we have $B_{tr_n}(x_0)\subseteq B_{tr_m}(x_0)$, and hence 
\begin{equation}
	\int_{B_{\frac{R}{2}}(x_0)\setminus B_{tr_m}(x_0)}\frac{\mm\big(B_{\frac{3tr_n}{2}}(x)\big)}{\left(\frac{3tr_n}{2}\right)^N}\,\d\mm(x)\leq 	C(\beta,p,N,C_2)\mm\big(B_R(x_0)\big).
\end{equation}
Letting $n\rightarrow+\infty$, applying the monotone convergence theorem, and then letting $m\rightarrow+\infty$, we obtain
\begin{equation*}
	\frac{1}{\mm\big(B_R(x_0)\big)}\int_{B_{\frac{R}{2}}(x_0)}f(x)\,\d\mm(x)\leq \frac{C(\beta, p, N)}{\mm\big(B_R(x_0)\big)}\int_{B_{\frac{R}{2}}(x_0)}g(x)\,\d\mm(x)\leq C(\beta,p,N,C_2).
\end{equation*}
By the doubling condition,
\begin{equation*}
\frac{1}{\mm\big(B_\frac{R}{2}(x_0)\big)}\int_{B_{\frac{R}{2}}(x_0)}f(x)\,\d\mm(x)\leq \frac{\beta}{\mm\big(B_R(x_0)\big)}\int_{B_{\frac{R}{2}}(x_0)}f(x)\,\d\mm(x)\leq C(\beta,p,N,C_2).
\end{equation*}

Similarly, replacing $r$ with $r_n$ in \eqref{upp} and letting $n\rightarrow +\infty$ yields
%the monotone convergence theorem 
\begin{equation*}
	\frac{1}{\mm\big(B_R(x_0)\big)}\int_{B_R(x_0)}f(x)\,\d\mm(x)\geq C(C_1).
\end{equation*}
Letting $R\rightarrow 0^+$, we deduce from the Lebesgue differentiation theorem (cf. \cite[Section 3.4]{zbMATH06397370}) that, for $\mm$-a.e.\ $x_0\in X$,
\begin{equation*}
	C(C_1)\leq f(x_0)\leq C(\beta,p,N,C_2).
\end{equation*}

Finally, for any $0<r<1$, there exists $n\in \N$ such that $r_{n+1}<r\leq 2r_{n+1}$, and we have
\begin{equation*}
		\frac{\mm\big(B_{r_{n+1}}(x)\big)}{2^{N}r_{n+1}^N}\leq \frac{\mm\big(B_r(x)\big)}{r^N}\leq	\frac{\mm\big(B_{2r_{n+1}}(x)\big)}{r_{n+1}^N}\leq 	\frac{\beta\mm\big(B_{r_{n+1}}(x)\big)}{r_{n+1}^N}.
\end{equation*}
Letting $r\rightarrow 0^+$, and hence also $r_{n+1}\rightarrow 0^+$, we obtain
\begin{equation*}
	C(N,C_1)\leq 2^{-N}f(x)\leq \theta_{N}^-(x)\leq \theta_{N}^+(x)\leq \beta f(x)\leq C(\beta,p,N,C_2),
\end{equation*}
which completes the proof.
\end{proof}

%\begin{remark}
%   In fact, the lower bound for the density function follows directly from Proposition~\ref{2.166} when \(N=\frac{\log \beta}{\log 2}\).
%\end{remark}

\begin{appendices}
\section{Proof of Counterexample \ref{counterex}}\label{appendix}

\begin{proof}
For $N\geq 1$, the base space $(X,\d,\mu)$ we use is an $N$-Ahlfors regular, unbounded, proper, and geodesic PI space. More precisely, for $N>1$, we use the Laakso space (cf. \cite[Theorem 2.7]{zbMATH01443217}), whereas for $N=1$, we take the real line $\R$ equipped with the Euclidean metric and Lebesgue measure. Such a space is necessarily a Lipschitz differentiability space (cf. \cite{MR1708448}). Recall that $N$-Ahlfors regularity means that there exist positive constants $c_1$ and $c_2$ such that
\begin{equation}\label{AR}
c_1 r^N\leq\mu(B_r(x))\leq c_2r^N, \quad \forall\, x\in X,\,\,  r>0.
\end{equation}

We divide the construction into four steps.

\paragraph{Step 1.}
Fix $a_0\in X$, $R>0$, and denote $r_k=2^{-k}$, $\rho_k=2^{-k^2}$ for $k\in\N^+$. Choose $K_0\in\N$ sufficiently large so that
\begin{equation}\label{3.43}
\sum_{k\geq K_0} 2c_2(20r_k)^N<\mu(B_R(a_0))\quad\text{and} \quad r_k\leq \frac{R}{2},\quad \forall\, k\geq K_0.
\end{equation}
We claim that there exist two sequences $\{x_k\}_{k\geq K_0}$ and $\{y_k\}_{k\geq K_0}$ satisfying
\begin{itemize}
    \item 
    \begin{equation}\label{eq:cond_a}
        \d(x_k,y_k) = 3r_k;
    \end{equation}
    \item 
    \begin{equation}\label{eq:cond_b}
        \forall\, j<k, \quad \d(x_k,x_j) \geq 20 r_j, \quad \d(x_k,y_j) \geq 20r_j;
    \end{equation}
    \item 
    \begin{equation}\label{eq:cond_c}
        \forall\, j<k, \quad \d(y_k,x_j) \geq 10 r_j, \quad \d(y_k,y_j) \geq 10r_j.
    \end{equation}
\end{itemize}

Proceeding by induction, assume that we have already constructed the points for $K_0,\dots,k-1$. For $k\in\N$, define
\begin{equation*}
F_k:=\bigcup_{K_0\leq j\leq k-1}\left(B_{20r_j}(x_j)\cup B_{20r_j}(y_j)\right).
\end{equation*}
By \eqref{AR} and \eqref{3.43}, we have $\mu(B_R(a_0)\cap F_k)<\mu(B_R(a_0))$, and hence there exists $x_k\in B_R(a_0)\setminus F_k$ satisfying \eqref{eq:cond_b}.

Since $(X,\d)$ is a geodesic space, there exists $y_k\in X$ such that $\d(x_k,y_k)=3r_k$. Moreover, by the triangle inequality, \eqref{eq:cond_c} holds.

\paragraph{Step 2.}
Denote $S_k=B_{\rho_k}(x_k)$ and $Y_k=B_{\rho_k}(y_k)$ for $k\geq K_0$. Since $\mu$ is a Radon measure, for any $x \in X$, the spherical boundary $\partial B_r(x)$ has positive $\mu$-measure for at most countably many radii $r>0$. Thus, by slightly perturbing $\rho_k$ if necessary, we may assume without loss of generality that $\mu(\partial S_k)=\mu(\partial Y_k)=0$. Note that
\begin{equation}\label{r_k}
\d(x,y)\leq  \d(x_k,y_k)+2\rho_k\leq 4r_k,\quad \forall \, x\in S_k,\, y\in Y_k.
\end{equation}
Moreover, by \eqref{eq:cond_a}, \eqref{eq:cond_b}, and \eqref{eq:cond_c}, one checks that
\begin{equation}
S_k\cap Y_k=\emptyset,\quad\forall\, k\geq K_0,\quad\text{and}\quad (S_k\cup Y_k)\cap(S_j\cup Y_j)=\emptyset,\quad \forall \, k\neq j.
\end{equation}

Define the Radon measure $\mm=\omega \mu$ with
\begin{equation}\label{omega}
\omega(x)=1+\sum_{k\geq K_0}\frac{1}{k^2\mu(S_k)}\chi_{S_k}(x)+\sum_{k\geq K_0}\frac{1}{k^2\mu(Y_k)}\chi_{Y_k}(x).
\end{equation}
Then $\mm(S_k)=\frac{1}{k^2}+\mu(S_k)$, $\mm(Y_k)=\frac{1}{k^2}+\mu(Y_k)$, and $\mm\ll \mu$.

Consider the space $(X,\d,\mm)$. Since $\mm\ll \mu$, the space $(X,\d,\mm)$ is also a Lipschitz differentiability space. Moreover, by \eqref{AR},
\begin{equation*}
\theta_N^+(x)=\varlimsup_{r\to 0^+}\frac{\mm\big(B_r(x)\big)}{r^N}\geq \varlimsup_{r\to 0^+}\frac{\mu\big(B_r(x)\big)}{r^N}\geq c_1>0,\quad \forall\, x\in X.
\end{equation*}

\paragraph{Step 3.}
Define $u_k(x)=\max\{0, r_k-\d(x,S_k)\}$ for $k\in\N$, and set
\[
u(x)=\sum_{k\geq K_0}u_k(x).
\]
Note that each $u_k$ is $1$-Lipschitz and $\supp u_k\subseteq B_{2r_k}(x_k)$. By \eqref{eq:cond_b}, we have
\begin{equation*}
\supp u_k\cap \supp u_j=\emptyset,\quad \forall \, k\neq j.
\end{equation*}
Since, by construction, $x_k\in B_R(a_0)\setminus F_k$, it follows that
\begin{equation*}
\supp u\subseteq \bigcup_{k\geq K_0}B_{2r_k}(x_k)\subseteq B_{2R}(a_0).
\end{equation*}
Moreover, one can check that $u$ is still $1$-Lipschitz, and hence $u\in \Lip_b(X,\d)$.

\paragraph{Step 4.}
On the one hand, denote
\begin{equation*}
Z_k:=\{x\in X: 0<\d(x,S_k)\leq r_k\},\quad \forall \, k\geq K_0.
\end{equation*}
Since $\mu(\partial S_k)=0$, we have $\Lip(u)=0$ for $\mm$-a.e.\ $x\in X\setminus\cup_{k\geq K_0}Z_k$. Moreover, for any $x\in \cup_{k\geq K_0}Z_k$, by \eqref{eq:cond_b} and \eqref{eq:cond_c}, it holds that
\begin{equation*}
\d\bigl(x, \cup_{j\geq K_0}(S_j\cup Y_j)\bigr)>0.
\end{equation*}
Hence, by \eqref{omega}, there exists a sufficiently small neighborhood of $x$ on which $\mm=\mu$. Consequently, by \eqref{AR}, we have
\[
\theta_N^+(x)\leq c_2,\qquad x\in \cup_{k\geq K_0}Z_k.
\]
Thus,
\begin{equation*}
\int_{X}\theta_N^+(x)\big({\rm Lip}(u)(x)\big)^p\,\d\mm(x)\leq\int_{\cup_{k\geq K_0}Z_k}\theta_N^+(x)\,\d\mu(x)\leq c_2^2\sum_{k\geq K_0}(2r_k)^N<+\infty.
\end{equation*}

On the other hand, by \eqref{eq:cond_b} and \eqref{eq:cond_c}, we have $u(y)=0$ for all $y\in Y_k$. Take $\lambda_k=4^{-\frac{N}{p}-1}r_k^{-\frac{N}{p}}$. For any $(x,y)\in S_k\times Y_k$, we have $|u(x)-u(y)|=r_k$, and by \eqref{r_k},
\begin{equation*}
|u(x)-u(y)|= \lambda_k (4r_k)^{\frac{N}{p}+1}\geq \lambda_k\big(\d(x,y)\big)^{\frac{N}{p}+1},
\end{equation*}
which implies $(x,y)\in E_{\lambda_k,u}$, and hence $S_k\times Y_k\subseteq E_{\lambda_k,u}$.

Therefore, recalling that $\mm(S_k)=\frac{1}{k^2}+\mu(S_k)$ and $\mm(Y_k)=\frac{1}{k^2}+\mu(Y_k)$, we obtain
\begin{equation*}
\varlimsup_{\lambda\rightarrow +\infty} \lambda^p (\mm\times \mm)({E_{\lambda,u}})\geq \lim_{k\rightarrow +\infty} \lambda_k^p (\mm\times \mm)({E_{\lambda_k,u}})\geq \lim_{k\rightarrow +\infty}\lambda_k^p\mm(S_k)\mm(Y_k)=+\infty.
\end{equation*}
This completes the proof.
\end{proof}

\end{appendices}

	\bibliographystyle{siam} %按作者姓氏首字母顺序排序
	\bibliography{Reference.bib} %调用参考文献文档

    \bigskip
    
\noindent \textbf{Declaration.}
The authors declare that there is no conflict of interest. The manuscript has no associated data.

\end{document}